\documentclass[a4paper]{article}
\usepackage{amssymb,amsthm,amsmath,graphicx,amsfonts}
\usepackage{mathrsfs}
\usepackage{multirow}
\usepackage{graphicx,epsfig}
\usepackage{color}
\usepackage{enumerate,enumitem}
\usepackage{empheq}
\usepackage{cases}
\usepackage{cite}
\usepackage{mathrsfs}
\usepackage{amsgen}
\usepackage{amscd}
\usepackage[ruled,vlined]{algorithm2e}
\usepackage{float}
\usepackage{booktabs}
\usepackage{amsmath}
\usepackage{latexsym}
\usepackage{amssymb,amsmath,graphicx,amsfonts,euscript}
\usepackage{amsfonts,bm,hyperref,subcaption}
\usepackage{enumerate}
\newtheorem{remark}{\bf Remark}[section]

\topmargin \dimexpr -0.11in
\allowdisplaybreaks

\newtheorem{Theorem}{Theorem}[section]

\newtheorem{lemma}[Theorem]{Lemma}

\def\d{{\mathrm d}}

\def\le{\leqslant}
\def\ge{\geqslant}
\def\Omega{\varOmega}
\def\Delta{\varDelta}
\def\bex{\begin{exercise}\upshape}
\def\eex{\end{exercise}}

\newcommand{\Bw}{{\bm{w}}}

\usepackage{geometry}

\title{An inertial minimal-deformation-rate framework for shape optimization\thanks{This work is supported in part by National Natural Science Foundation of China under grants (No. 12525111, No. 12494550, No. 12494555 and No. 12401534).
}} 
\author{Falai Chen\thanks{School of Mathematics, University of Science and Technology of China. 
E-mail address: chenfl@ustc.edu.cn  
}
\and Buyang Li\thanks{Department of Applied Mathematics, The Hong Kong Polytechnic University, Hong Kong. 
E-mail address: buyang.li@polyu.edu.hk, jiajienumermath.li@polyu.edu.hk, claire.tang@polyu.edu.hk }
\and Jiajie Li$^\ddagger$
\and Rong Tang$^\ddagger$
}

\begin{document}

\maketitle


\begin{abstract}
We propose a robust numerical framework for PDE--constrained shape optimization and Willmore-driven surface hole filling. 
To address two central challenges—slow progress in flat energy landscapes, which can trigger premature stagnation at suboptimal configurations, and mesh deterioration during geometric evolution—we couple a second--order inertial flow with a minimal--deformation--rate (MDR) mesh motion strategy. 
This coupling accelerates convergence while preserving mesh quality and thus avoids remeshing. 
To further enhance robustness for non-smooth or non-convex initial geometries, we incorporate surface--diffusion regularization within the Barrett--Garcke--N\"urnberg (BGN) framework. 
Moreover, we extend the inertial MDR methodology to Willmore-type surface hole filling, enabling high-order smooth reconstructions even from incompatible initial data. 
Numerical experiments demonstrate markedly faster convergence to lower original objective values, together with consistently superior mesh preservation throughout the evolution.

\end{abstract}

\section{Introduction}

Shape optimization constrained by partial differential equations (PDEs) is a fundamental problem in computational science and engineering, with applications ranging from aerodynamics~\cite{MohammadiPironneau2010} to structural design~\cite{AllaireDapognyJouve2021,HaugChoiKomkov1986}. The central topic is to determine a domain \(\Omega\subset\mathbb{R}^d\) (\(d\in\{2,3\}\)) with boundary \(\Gamma:=\partial\Omega\) that minimizes a original objective functional \(J(\Gamma)=\int_{\Omega} j(x,u)\), subject to a PDE constraint~\cite{DelZol,SokolowskiZolesio1992,HenrotPierre2018,HaslingerMakinen2003,HaslingerNeittaanmaki1996}. A prototypical example is the Poisson equation--constrained shape design problem:
\begin{equation}\label{shapeOptPoisson}
\min_{\Omega\in \mathcal{O}_{\rm ad}} J(\Gamma)=\int_\Omega \frac{1}{2}|u-u_d|^2
\quad \text{subject to}\quad 
\left\{
\begin{aligned}
-\Delta u + u & = f && \text{in } \Omega,\\
u&=0 && \text{on } \Gamma,
\end{aligned}
\right.
\end{equation}
where \(u_d\) denotes prescribed observation data and \(\mathcal{O}_{\rm ad}\) is the admissible set of Lipschitz domains. The practical relevance of such problems has stimulated extensive research on numerical methods for evolving geometries~\cite{Burman2017,Guo2020,Dede2012,Wang2018,Morin2012,Eppler2006} as well as rigorous convergence analyses for time-dependent domains~\cite{Ma2022,Boffi2004,Badia2006,ElliottRanner2021,Edelmann2022,Gastaldi2001,Nobile2001,ElliottStyles2012,ElliottVenkataraman2015,KovacsGuerra2018,LiXiaYang2023,GongLiRao2024}.

A standard solution strategy relies on shape gradient flows~\cite{Burger2003,Gournay2006,DoganMorinNochettoVerani2007,GongZhu2021}, in which the descent velocity \(\tilde{\bm w}\) is defined via the Riesz representation in a Hilbert space \(\mathbb{H}\) by
\begin{align}\label{descentVelocity}
  b(\tilde{\bm w}, \bm v) = -\,\mathrm{d}J(\Gamma; \bm v)
  \quad \forall\,\bm v \in \mathbb{H},
\end{align}
where \(b(\cdot,\cdot)\) denotes the inner product on \(\mathbb{H}\) and \(\mathrm{d}J(\Gamma;\cdot)\) is the Eulerian derivative. The choice of \(\mathbb{H}\) dictates the regularity of the flow; common examples include \(\mathbb{H}=\textbf{H}^1(\Omega)\), \(\mathbb{H}=\textbf{L}^2(\Gamma)\), or \(\mathbb{H}=\textbf{H}^1(\Gamma)\)~\cite{SchulzCMAM2016,Schulz2016SICON}. In two dimensions, conformal-mapping-based flows have also been explored to enhance mesh quality for shape minimization~\cite{CR2018}. Let \(\Omega(t):=\tilde{\bm X}(t)(\Omega^0)\) denote the evolving domain induced by the flow map \(\tilde{\bm X}(t):\Omega^0\to\mathbb{R}^d\). The classical first-order shape gradient flow then reads
\begin{subequations}\label{first-order-flow}
\begin{align}
\partial_t\tilde{\bm X}(t) &= \tilde{\bm w}(t)\circ\tilde{\bm X}(t)
&& \text{in } \Omega^0, \label{eq:flo-map}\\
b(\tilde{\bm w}(t),\bm v) &= -\,\mathrm{d}J(\Gamma(t);\bm v)
&& \forall\,\bm v\in \mathbb{H}(t).
\end{align}
\end{subequations}

Despite their simplicity, first-order dynamics~\eqref{first-order-flow} often converge slowly and cannot escape shallow local basins. Since the Riesz isomorphism $\mathcal R:\mathbb{H}\to\mathbb{H}^\ast$ gives $\tilde{\bm w}=-\mathcal R^{-1}\mathrm{d}J$, the velocity vanishes as $\mathrm{d}J\to0$, causing stagnation in flat regions and premature convergence. This motivates inertial effects, which introduce kinetic memory to traverse plateaus and escape shallow minima.

Conceptually, second-order (inertial) flows are inspired by accelerated dynamics in convex optimization~\cite{AttouchMP2018,Edvardsson2012} and can be viewed as continuous-time analogues of Nesterov-type acceleration~\cite{Nesterov1983,Boyd2016}. This perspective has recently been extended to infinite-dimensional problems, including nonlinear second-order PDEs for Bose--Einstein condensates~\cite{ChenDong2023JCP} and dissipative inertial dynamics on Sobolev spaces for non-convex energies~\cite{2ndFlowSISC2025}.

Motivated by these developments, we consider an inertial second-order evolution for PDE-constrained shape optimization of the form
\begin{equation}\label{eq:inertial-shape-flow}
  \epsilon_0\,a\big(\partial_{tt}^\bullet \tilde{\bm X}(t),\, \partial_{t}^\bullet \tilde{\bm X}(t),\,\bm v\big)
  +\eta(t)\,b\big(\partial_t^\bullet \tilde{\bm X}(t),\,\bm v\big)
  = - \mathrm dJ\big(\Gamma(t);\,\bm v\big),
  \qquad \forall\,\bm v\in \mathbb{H}(t),
\end{equation}
where \(\partial_{tt}^\bullet \tilde{\bm X}(t)\) denotes the second material derivative, \(\epsilon_0\ge0\) and \(\eta(t)>0\) are inertia and damping parameters, and the first-order gradient flow~\eqref{first-order-flow} is recovered by setting \(\epsilon_0\equiv0\) and \(\eta(t)\equiv1\). The term \(a(\cdot,\cdot,\cdot)\) accounts for Reynolds-type transport effects induced by the evolving geometry. This formulation admits the mechanical energy
\begin{equation}\label{eq:mechanical-energy}
  \mathcal H(t)
  := J\big(\Gamma(t)\big)
  + \frac{\epsilon_0}{2}\,\|\partial_t^\bullet \tilde{\bm X}(t)\|_{\textbf{L}^2(\Omega(t))}^2,
\end{equation}
which satisfies the energy dissipation property \(\frac{\mathrm d}{\mathrm dt}\mathcal H(t)\le0\). While the inertial term allows for temporary non-monotonicity of the original objective functional \(J(\Gamma(t))\), 
the damping term \(\eta(t)\) ensures energy dissipation and suppresses spurious oscillations. 
At the discrete level, the resulting scheme is linear at each time step, computationally efficient, 
and exhibits significantly faster convergence in practice compared to the corresponding first-order flow~\eqref{first-order-flow}, 
as demonstrated in Section~\ref{sec:numerics}.

However, inertia alone does not prevent mesh degradation, which can occur for both first- and second-order flows. 
In the continuous setting, the shape of \(\Omega(t)\) is determined only by the normal boundary velocity; the tangential boundary motion and the extension into the interior are shape-invariant degrees of freedom, yet crucial for mesh-quality control. 
Standard interior extensions, such as harmonic extension~\cite{Boffi2004,Edelmann2022,LiXiaYang2023}, may lead to severe mesh distortion or even element entanglement under large deformations. 
This motivates incorporating a dedicated mesh-motion strategy to preserve mesh quality throughout the evolution and avoid remeshing.

To address this mesh-quality issue while respecting the prescribed shape evolution, we adopt and extend the Minimal Deformation Rate (MDR) framework~\cite{BaiHuLi2024}. 
In this framework, we impose the normal component of the boundary velocity as a hard constraint, namely
\(\bm w\cdot\bm n|_{\Gamma(t)}=\tilde{\bm w}\cdot\bm n\),
where \(\tilde{\bm w}\) is the velocity prescribed by the inertial flow~\eqref{eq:inertial-shape-flow}, with \(\tilde{\bm w}=\partial_t^\bullet \tilde{\bm X}(t)\). 
The bulk velocity field \(\bm w\) is then obtained by minimizing the instantaneous deformation-rate energy
\begin{equation}\label{MDR_intro}
  \frac12 \int_{\Omega(t)} \bigl|\varepsilon(\bm w)\bigr|^2,
  \quad \text{subject to } \bm w\cdot\bm n|_{\Gamma(t)}=\tilde{\bm w}\cdot\bm n,
\end{equation}
where \begin{align}\label{stress_tensor}
\varepsilon(\bm w):=\tfrac12(\nabla\bm w+\nabla\bm w^{\top})
\end{align} denotes the symmetric strain-rate tensor. 
This constrained minimization yields a linear variational problem at each time step, requiring only a single linear solver. 
By directly minimizing the deformation rate, the proposed method maintains high mesh quality throughout the evolution. It significantly outperforms both first- and second-order flows as well as classical approaches based on standard harmonic extension, offering a robust solution where specialized two-dimensional remedies, such as conformal mappings~\cite{CR2018}, do not readily generalize to three dimensions.


Despite their favorable theoretical properties, both classical gradient flows and the proposed inertial MDR framework remain prone to numerical instabilities when applied to non-smooth (e.g., cornered) or non-convex initial geometries. Such geometric irregularities typically induce severe mesh entanglement near corners and reentrant boundaries, often precipitating solver failure. Consequently, much of the existing literature has focused primarily on convex~\cite{ZhuEig,GongLiRao2024,Hiptmair2015BIT} or smooth star-shaped domains~\cite{Antunes2017}.

To address these robustness challenges, we regularize the second-order flow~\eqref{eq:inertial-shape-flow} by incorporating a surface-diffusion term. Surface diffusion, which can be interpreted as the \(H^{-1}\)-gradient flow of the area functional~\cite{Mullins1957,BaenschMorinNochetto2005,BaoZhao2021,LiBao2021}, is particularly suitable in this context, as it effectively smooths geometric singularities while preserving the enclosed volume—an essential property for PDE-constrained shape optimization. We implement this regularization by modifying the normal-velocity constraint as
\begin{align}\label{eq:boundary-regularization}
\hat{\bm w}\cdot\bm n = \bigl(\tilde{\bm w} + \alpha\,\Delta_\Gamma H\,\bm n\bigr)\cdot\bm n
\quad \text{on } \Gamma(t),
\end{align}
where \(\tilde{\bm w}\) is the velocity prescribed by the inertial flow~\eqref{eq:inertial-shape-flow}, and \(H\) is the mean curvature. The regularization parameter is adapted dynamically as \(\alpha = \|\tilde{\bm w}\|_{{\bf L}^2(\Gamma(t))}\), so that geometric smoothing is dominant in the early stages to resolve singularities and gradually diminishes as the flow approaches equilibrium.

To discretize the regularized flow, we employ the Barrett--Garcke--N\"urnberg (BGN) framework~\cite{BarrettGarckeNuernberg2007,BarrettGarckeNuernberg2008JCP,BarrettGarckeNuernberg2008SISC}. 
This formulation provides a reliable discretization of the fourth-order surface diffusion term while naturally generating an artificial tangential velocity that preserves mesh quality throughout the evolution. 
To the best of our knowledge, this is the first integration of surface-diffusion regularization into PDE-constrained shape optimization. 
The resulting regularized boundary velocity is then extended into the bulk domain using the MDR approach described above.

Beyond PDE-constrained shape design, we extend this inertial MDR framework to Willmore-driven surface hole filling to reconstruct high-quality interior patches from incomplete geometries. Standard area minimization via mean curvature flow is often inadequate for this task, as it typically yields only $G^0$ continuity (positional continuity)—ensuring the patches merely connect—without enforcing the $G^1$ continuity (tangent plane alignment) required for a visually smooth transition. We therefore employ Willmore flow to explicitly promote this higher-order geometric regularity. In contrast to the parametric finite element approach in~\cite{garcke2025energystableparametricfiniteelement}, we develop a second-order inertial formulation coupled with MDR mesh motion. This strategy avoids auxiliary geometric evolution equations, thereby mitigating long-time error accumulation while maintaining robustness even for incompatible or non-smooth initial data.



The remainder of the paper is organized as follows. Section~2 introduces four representative shape optimization models and presents the corresponding Eulerian derivatives. Section~3 presents the proposed second-order inertial flow for PDE-constrained shape optimization. Section~4 develops the inertial BGN--MDR framework, combining surface-diffusion regularization (via the BGN formulation) with MDR-based interior velocity extension. Section~5 adapts the inertial MDR framework to Willmore-driven surface hole filling, with particular emphasis on clamped boundary conditions and incompatible initial data. Section~6 reports numerical results demonstrating accelerated convergence and robust mesh-quality preservation. Section~7 concludes with summary remarks and directions for future research.
 
\section{Representative models of shape optimization}
In this section, we introduce four representative shape-optimization models. 
The first three are classical PDE-constrained problems—shape reconstruction, drag minimization for Stokes flow, and eigenvalue optimization for elliptic operators. 
The fourth model concerns Willmore-driven surface hole filling, which we view as a geometric variational problem without bulk PDE constraints.

\subsection{Shape reconstruction}\label{sec.ShapeRecon}
The Eulerian derivative for the objective functional in problem \eqref{shapeOptPoisson} is given by (see \cite[eqs.~(2.9)--(2.10)]{Hiptmair2015BIT}):
\begin{equation}
\begin{aligned}
{\rm d}J(\Gamma,u,p;\bm v)=& \int_\Omega \nabla u\cdot (\nabla \bm v+ \nabla \bm v^{\rm T})\nabla p 
- f \nabla p \cdot \bm v -(u-u_d)\nabla u_d\cdot \bm v \\
&+\int_\Omega \Big(\tfrac{1}{2}|u-u_d|^2-\nabla u\cdot \nabla p-up\Big)\nabla\cdot \bm v,
\end{aligned}
\end{equation}
where $\bm v\in \textbf{H}^1(\Omega
)$ is the deformation field, and the adjoint variable $p$ is the solution to the following boundary value problem:
\begin{equation}\label{ShapeReconPDE}
\begin{aligned}
-\Delta p+p&=u-u_d \quad &&{\rm in}\ \Omega, \qquad p=0 &&{\rm on}\ \Gamma.
\end{aligned}
\end{equation}
Applications of this model include Electrical Impedance Tomography (EIT) and the identification of pollutant sources or contaminated regions via inverse diffusion.

\subsection{Drag minimization in Stokes flow} \label{secdrag}
Drag minimization~\cite{MohammadiPironneau2010} plays a critical role in computational fluid dynamics, with applications ranging from aerodynamic wing design to medical bypass anastomoses and pipe flow optimization. Let $\mu>0$ denote the fluid viscosity. We partition the boundary as $\Gamma=\Gamma_i\cup\Gamma_w\cup\Gamma_o$, representing the inlet $\Gamma_i$, wall $\Gamma_w$, and outlet $\Gamma_o$, respectively. Given Dirichlet data $\bm u_d\in \mathbf{H}^{1/2}(\Gamma_i)$ and a source term $\bm f \in \textbf{H}^{-1}(\Omega)$, the velocity $\bm u$ and pressure $p$ satisfy the steady Stokes equations:
\begin{equation}\label{StokesEqu}
\begin{aligned}
-\mu \Delta \bm u+ \nabla p&= \bm f &&{\rm in}\ \Omega, \qquad \nabla\cdot \bm u=0 &&{\rm in}\ \Omega,\\
\bm u&=\bm u_d &&{\rm on}\ \Gamma_i,\qquad\quad \  \bm u=\bm 0 &&{\rm on}\ \Gamma_w,\\
\mu \nabla\bm u\,\bm n - p{\rm I}\bm n&=0 &&{\rm on}\ \Gamma_o,
\end{aligned}
\end{equation}
where $\bm n$ is the unit outward normal and ${\rm I}\in \mathbb{R}^{d\times d}$ is the identity matrix. For the drag minimization problem subject to a volume constraint:
\begin{equation}\label{dragMinimize}
\min_{\Omega\in\mathcal{O}_{\rm ad}, \vert\Omega\vert={\rm const}}J(\Gamma)=\int_\Omega\mu|\nabla \bm u|^2
\end{equation}
constrained by \eqref{StokesEqu}, the Eulerian derivative is (see \cite[eq.~(36)]{ZhuCMAME}):
\begin{equation}\label{dJStokes}
\begin{aligned}
{\rm d}J(\Gamma,\bm u,p;\bm v)=\int_\Omega \mu\Big(|\nabla \bm u|^2\nabla\cdot \bm v
- \nabla \bm u^{\rm T}:(\nabla \bm v+\nabla \bm v^{\rm T})\nabla \bm u^{\rm T}\Big)
+ 2\bm v\cdot \nabla \bm u^{\rm T}(\bm f-\nabla p).
\end{aligned}
\end{equation}
Since the objective functional in \eqref{dragMinimize} leads to an adjoint equation that coincides with the state equation, no separate adjoint variable is required.
 The volume constraint $|\Omega|=\text{const}$ is enforced via a Lagrange multiplier to prevent trivial scaling of the domain.

\subsection{Eigenvalue minimization}
Consider the Laplace eigenvalue problem:
\begin{equation}\label{EigenPoisson}
-\Delta u = \lambda u \quad \text{in } \Omega, \qquad u=0 \quad \text{on } \Gamma.
\end{equation}
This problem admits an increasing sequence of eigenpairs $\{(\lambda_\ell, u_\ell)\}_{\ell=1}^\infty \subset \mathbb{R} \times H_0^1(\Omega)$ with associated $L^2$-orthonormal eigenfunctions. We aim to minimize the $\ell$-th eigenvalue $\lambda_\ell$ subject to a volume constraint. The variational characterization is provided by the Courant--Fischer min--max principle:
\begin{equation}\label{MinEigen}
\min_{\Omega \in \mathcal{O}_{\text{ad}}, |\Omega|=1} J(\Omega) = \lambda_\ell, \quad 
\lambda_\ell := \min_{\substack{S_\ell \subset H^1_0(\Omega) \\ \dim S_\ell = \ell}} \max_{0 \neq v \in S_\ell} \frac{\int_\Omega |\nabla v|^2 }{\int_\Omega v^2 }.
\end{equation}
Since eigenvalues scale according to $\lambda_\ell(s\Omega) = s^{-2} \lambda_\ell(\Omega)$, unconstrained minimization would result in unbounded domain expansion; hence, the volume constraint is essential.

Assuming the eigenvalue $\lambda_\ell$ remains simple along the shape evolution, the corresponding Eulerian derivative in volume formulation is given by (cf. \cite[eq.~(6)]{ZhuEig}):
\begin{equation}\label{dJEigen}
{\rm d}J(\Omega, u_\ell, \lambda_\ell; \bm{v}) = \int_\Omega \text{div}(\bm{v}) \big( |\nabla u_\ell|^2 - \lambda_\ell u_\ell^2 \big) - \int_\Omega \nabla u_\ell \cdot (\nabla \bm{v} + \nabla \bm{v}^{\rm T}) \nabla u_\ell.
\end{equation}

In the presence of multiple eigenvalues, classical shape differentiability generally fails as the objective functional is only directionally differentiable. In practice, following the strategy in \cite{ZhuEig}, we compute the directional derivatives associated with the eigenspace corresponding to the multiple eigenvalue and employ their ensemble average to define a stable descent direction.

\subsection{Surface hole filling driven by Willmore flow}

Surface hole filling is a fundamental problem in Computer Aided Geometric Design (CAGD) and geometric modeling~\cite{hoschek1996fundamentals}. We consider a surface consisting of a known exterior region and a missing interior patch. Given the exterior surface and its boundary curve $\partial\Gamma$, the goal is to reconstruct the missing patch with $G^1$ geometric continuity across $\partial\Gamma$, ensuring tangent plane alignment and high surface quality.
A straightforward approach is to seek a surface of minimal area, leading to mean curvature flow. However, prescribing the boundary curve guarantees only $G^{0}$ (positional) continuity; since the surface normals along $\partial\Gamma$ are unconstrained, $G^{1}$ continuity is not achieved.

To promote higher-order geometric regularity, we consider the bending energy defined as the Dirichlet energy of the unit normal field
\(
\mathcal{E}(\Gamma) := \int_\Gamma |\nabla_\Gamma \bm n|^2 .
\)
This energy penalizes rapid variations of the tangent plane and promotes a smoothly varying surface geometry. It is closely related to the classical Willmore energy
\[
E(\Gamma):=\int_\Gamma H^{2},
\]
where $H$ denotes the mean curvature. Using the pointwise identity
\(
|\nabla_\Gamma \bm n|^{2}=H^{2}-2K,
\)
where \(K\) denotes the Gaussian curvature, we obtain
\(
\int_\Gamma |\nabla_\Gamma \bm n|^{2}
=\int_\Gamma H^{2} -2\int_\Gamma K.
\)
For a compact orientable surface \(\Gamma\) with boundary \(\partial\Gamma\), the Gauss--Bonnet theorem
(Section~4.5 in~\cite{carmo}) states that
\(
\int_\Gamma K\,\mathrm dA + \int_{\partial\Gamma} k_g\,\mathrm ds = 2\pi\chi(\Gamma),
\)
where \(k_g\) denotes the geodesic curvature of \(\partial\Gamma\) in \(\Gamma\).
If the boundary curve and its conormal (equivalently, the boundary tangent plane) are fixed,
then \(\int_{\partial\Gamma} k_g\,\mathrm ds\) is constant.
Moreover, the Euler characteristic of \(\Gamma\) is given by
\(\chi(\Gamma)=2-2g-b\), where \(g\) denotes the genus and \(b\) the number of boundary components
(Chapter~1 in~\cite{missy}; Chapter~6 in~\cite{JohnLEE}).
Since \(\chi(\Gamma)\) is a topological invariant, it follows that
\(\int_\Gamma K\) is constant within the admissible class.
Consequently,
\[
\int_\Gamma |\nabla_\Gamma \bm n|^{2}
= \int_\Gamma H^{2} + C,
\]
where \(C\) is a constant, and hence minimizing the Dirichlet energy
\(\int_\Gamma |\nabla_\Gamma \bm n|^{2}\)
is equivalent to minimizing the Willmore energy
\(\int_\Gamma H^{2}\).

Motivated by this equivalence, we formulate the hole-filling problem as a geometric evolution in which an initial spanning surface $\Gamma^{0}$ bounded by $\partial\Gamma$ evolves according to Willmore flow. The evolution is parametrized by a time-dependent map $\bm X(\cdot,t):\Gamma^{0}\to\mathbb{R}^{3}$ with $\Gamma(t)=\bm X(\Gamma^{0},t)$ and $\bm X(\cdot,0)=\mathrm{id}_{\Gamma^{0}}$, and is governed by
\begin{align}\label{willmore-flow}
\partial_t\bm X
=
\Big(
\Delta_{\Gamma(t)} H_{\Gamma(t)}
-\tfrac12\,H_{\Gamma(t)}\big(H_{\Gamma(t)}^{2}
-2|\nabla_{\Gamma(t)} \bm n_{\Gamma(t)}|^2\big)
\Big)\,\bm n_{\Gamma(t)}
\quad \text{on } \Gamma(t).
\end{align}
This purely normal velocity corresponds to the $L^2$--gradient flow of the Willmore energy in the interior, where $\bm n_{\Gamma(t)}$, $\nabla_{\Gamma(t)}\bm n_{\Gamma(t)}$, and $H_{\Gamma(t)}$ denote the unit normal, the Weingarten map, and the mean curvature, respectively.
To enforce $G^{1}$ continuity and ensure a seamless transition to the exterior surface, we impose the boundary conditions
\begin{subequations}\label{boundary}
\begin{align}
\bm X(p,t) &= p \qquad\text{for } p\in\partial\Gamma, \label{boundary_1}\\
\bm\mu(t) &= \bm\mu_{\mathrm{out}} \qquad\text{on } \partial\Gamma, \label{boundary_2}
\end{align}
\end{subequations}
where $\bm\mu(t)$ denotes the conormal to $\Gamma(t)$ along $\partial\Gamma$, and $\bm\mu_{\mathrm{out}}$ is the fixed conormal induced by the exterior surface. We do not require compatible initialization, that is, $\bm\mu(0)$ need not coincide with $\bm\mu_{\mathrm{out}}$; instead, the boundary conditions ensure that the conormal naturally relaxes to the prescribed configuration during the evolution.

\section{Inertial Gradient Flow for PDE-constrained Shape Optimization}
We formulate an inertial gradient flow for PDE-constrained shape optimization. Starting from the classical first-order $\mathbf{H}^1(\Omega(t))$ gradient flow, we incorporate an inertial contribution via the material time derivative of the descent velocity $\tilde{\bm w}$. This leads to a second-order-in-time evolution dynamics where the velocity $\tilde{\bm w}$ is governed not only by the instantaneous Eulerian derivative but also by the accumulated motion history encapsulated in the material derivative $\partial_t^\bullet  \tilde{\bm w}$.

The inertial contribution is scaled by a parameter $\epsilon_{0}>0$; its role is to introduce momentum into the shape evolution, thereby mitigating stagnation in nearly flat regions where the first--order descent velocity becomes small, and enabling accelerated convergence toward configurations of lower original objective value.

We first present the continuous variational formulation of the inertial flow and its associated structural properties, then derive a time discretization of the inertial term on evolving domains, and finally propose a computationally efficient fully discrete linear scheme. 
In the next section, we introduce an inertial--MDR coupling framework that maintains mesh quality without remeshing.

\subsection{${\bf H}^1(\Omega(t))$ gradient flow and domain evolution}\label{subsec:H1_flow}

Let $\Omega(t)\subset\mathbb{R}^d$ be a smoothly evolving domain with the boundary $\Gamma(t)=\partial\Omega(t)$.
The ${\bf H}^1(\Omega(t))$ shape gradient flow is defined implicitly through the following variational equation: Find the descent velocity $\tilde{\bm w}(t)\in{\bf H}^1(\Omega(t))$ such that
\begin{align}
\int_{\Omega(t)} \tilde{\bm w}(t) \cdot \bm \chi_{{\bm w}}
+  \nabla \tilde{\bm w}(t): \nabla \bm \chi_{{\bm w}} 
= -  {\rm d}J(\Gamma(t);\bm \chi_{{\bm w}}) 
\qquad \forall\, \bm \chi_{{\bm w}}	 \in {\bf H}^1(\Omega(t)).
\label{H1Descentflow}
\end{align}
Once the velocity field $\tilde{\bm w}(t)$ is determined via \eqref{H1Descentflow}, the domain $\Omega(t)$ is evolved as the image of the initial configuration $\Omega^0$ under a flow map $\widetilde{\phi}: \Omega^0 \times [0,T] \to \mathbb{R}^d$. Specifically, the flow map $\widetilde{\phi}$ is governed by the following relation:
\begin{align}\label{flow}
    \partial_t \widetilde{\phi}(p, t) = \tilde{\bm w}(\widetilde{\phi}(p, t), t) \quad \text{for } (p, t) \in \Omega^0 \times [0,T],
\end{align}
subject to the initial condition $\widetilde{\phi}(p, 0) = p$. The domain and its boundary are defined by
\[
\Omega(t):=\{\widetilde{\phi}(p,t):p\in\Omega^{0}\}, \qquad 
\Gamma(t):=\{\widetilde{\phi}(p,t):p\in\Gamma^{0}\},
\]
where $\Gamma^{0}:=\partial\Omega^{0}$ denotes the boundary of the initial domain. In particular, $\partial\Omega(t)=\Gamma(t)$.

\subsection{Inertial ${\bf H}^1(\Omega(t))$ gradient flow}\label{subsec:inertial_H1_flow}

To complement the standard gradient flow~\eqref{H1Descentflow}, we introduce inertia through the material time derivative of the descent velocity. 
Beyond the instantaneous steepest-descent mechanism, this inertial contribution provides dynamical persistence, alleviating slow-down when the shape gradient becomes small and improving robustness in nearly flat regions of the objective landscape. 
To preserve an energy-dissipation structure on the evolving domain, we employ a transport-consistent (Reynolds-corrected) formulation: in particular, we use the inertial form
\(
a(\partial_t^\bullet \tilde{\bm w},\tilde{\bm w},\bm\chi_{\bm w})
\),
which incorporates the geometric correction term associated with the Reynolds transport theorem. 
The strength of the inertial effect is controlled by \(\epsilon_0>0\), allowing the framework to be tuned for different models and numerical settings.

The inertial contribution is defined by the following form:
\begin{equation}
a(\partial_{t}^\bullet \tilde{\bm{w}}, \tilde{\bm{w}}, \bm{\chi}_w)
:= \int_{\Omega(t)} 
\Bigl[ \partial_t^{\bullet} \tilde{\bm{w}} 
+ \tfrac{1}{2} \tilde{\bm{w}} (\nabla \cdot \tilde{\bm{w}}) \Bigr] \cdot \bm \chi_w,
\end{equation}
while the $\textbf{H}^1(\Omega(t))$ inner product is represented by the bilinear form:
\begin{equation}
b(\tilde{\bm w}, \bm \chi_w)
:= \int_{\Omega(t)}
\bigl(\tilde{\bm w} \cdot \bm \chi_w + \nabla \tilde{\bm w} : \nabla \bm \chi_w \bigr).
\end{equation}
The resulting inertial $\textbf{H}^1(\Omega(t))$ gradient flow seeks a velocity $\tilde{\bm{w}}(t) \in \textbf{H}^1(\Omega(t))$ such that
\begin{equation}\label{2ndflowExp}
\epsilon_0 a(\partial_t^\bullet \tilde{\bm{w}}, \tilde{\bm{w}}, \bm{\chi}_w) + \eta(t) b(\tilde{\bm{w}}, \bm{\chi}_w) = -{\rm d}J(\Gamma(t); \bm \chi_w), \quad \forall \bm{\chi}_w \in \textbf{H}^1(\Omega(t)),
\end{equation}
where $\eta(t)$ is a time-dependent damping coefficient. 
Following the damping strategy in \cite{ZhuEig}, we set
\begin{equation}
\eta(t) = \frac{c}{t + t_0} \quad (c > 0, \, t_0 > 0),
\end{equation}
which renders the dynamics initially overdamped and thus dominated by gradient descent. 
As $t$ increases, the damping $\eta(t)$ decreases, so that the dynamics gradually transitions from an overdamped regime to one in which inertial effects are more visible. In this way, the evolution can traverse nearly flat regions of the objective landscape without stagnation, while the remaining oscillations are damped out asymptotically, which in turn facilitates convergence toward configurations of lower original objective value.
To establish the energy dissipation property of the mechanical energy, we first recall the following transport theorem.


\begin{lemma}[{\cite[Lemma~5.7]{Walker2015}}]\label{dtIntFlemma}
Let $\Omega(t)$ be a smoothly evolving domain with velocity field $\tilde{\bm w} \in {\bf W}^{1,\infty}(\Omega(t))$. Then, for any sufficiently regular function $f:\Omega(t)\times [0,T]\to \mathbb{R}$, it holds that
\begin{equation}
\frac{{\rm d}}{{\rm d}t} \int_{\Omega(t)} f
= \int_{\Omega(t)} \partial_t^{\bullet} f
+ f\, \nabla \cdot \tilde{\bm w},
\end{equation}
where $\partial_t^{\bullet} f := \partial_t f + \tilde{\bm w} \cdot \nabla f$ denotes the material derivative of $f$ along $\tilde{\bm w}$.
\end{lemma}

\begin{lemma}[Energy dissipation for the inertial gradient flow]
Let \(\Omega(t)\) denote the evolving domain with boundary \(\Gamma(t)\), whose motion is driven by the velocity field \(\tilde{\bm w}\) satisfying~\eqref{2ndflowExp}. We define the associated mechanical energy by
\begin{align}\label{modified-energy}
\mathcal H(t) :=  J\big(\Gamma(t)\big) + \frac{\epsilon_0}{2}\,\|\tilde{\bm w}(t)\|_{{\bf L}^2(\Omega(t))}^2
 =  \int_{\Omega(t)}\!\Bigl(j(x,u) + \tfrac{\epsilon_0}{2}\,|\tilde{\bm w}|^2\Bigr).
\end{align}
The following dissipation property is satisfied for all $t>0$:
\begin{equation}\label{energyLaw}
\frac{{\rm d}}{{\rm d}t}\,\mathcal H(t)
 =  -\,\eta(t)\,\Bigl(\|\tilde{\bm w}\|_{{\bf L}^2(\Omega(t))}^2 + \|\nabla \tilde{\bm w}\|_{{\bf L}^2(\Omega(t))}^2\Bigr) \le 0.
\end{equation}
\end{lemma}

\begin{proof}
Applying Lemma~\ref{dtIntFlemma} with $f=\tfrac{\epsilon_0}{2}|\tilde{\bm w}|^2 + j(x,u)$ yields
\begin{equation}\label{EnerS1}
\begin{aligned}
\frac{{\rm d}}{{\rm d}t}\int_{\Omega(t)}\tfrac{\epsilon_0}{2}\vert \tilde{\bm w}\vert^2+j(x,u)
&= \int_{\Omega(t)}\Bigl(\epsilon_0 \tilde{\bm w} \cdot \partial_t^{\bullet}\tilde{\bm w} 
+ \tfrac{\epsilon_0}{2}\vert \tilde{\bm w}\vert^2 \nabla \cdot \tilde{\bm w}\Bigr)
+ {\rm d}J(\Gamma(t);\tilde{\bm w}).
\end{aligned}
\end{equation}
Then, inserting the weak formulation \eqref{2ndflowExp} with the choice $\bm \chi_w=\tilde{\bm w}$ leads directly to
\[
\frac{{\rm d}}{{\rm d}t} \mathcal H(t) = -\eta(t) \|\tilde{\bm w}\|_{{\bf L}^2(\Omega(t))}^2 - \eta(t)\|\nabla \tilde{\bm w}\|_{{\bf L}^2(\Omega(t))}^2,
\]
which establishes \eqref{energyLaw}.
\end{proof}

\subsection{Linear time discretization of the inertial term}\label{subsec:transport_discretization_inertial}

For the time discretization of the inertial term $a(\partial_{t}^\bullet \tilde{\bm{w}}, \tilde{\bm{w}}, \bm{\chi}_{\bm w})$, we adopt the  approach developed in~\cite{Gao2024_rotation}.
Let $[0,T]$ be the time interval partitioned into discrete levels $t_n = n \tau$ for $n = 0, 1, \ldots, N$, where $\tau > 0$ denotes the time step size. By applying the transport theorem, the following identity is established for any test function $\bm \chi_{\bm w} \in \mathbf{L}^2(\Omega^n)$:
\begin{align}\label{identity-1}
 a(\partial_{t}^\bullet \tilde{\bm{w}}, \tilde{\bm{w}}, \bm{\chi}_{\bm w}) &= \int_{\Omega^n} \partial_t^\bullet \tilde{\bm{w}} \cdot \bm \chi_{\bm w} 
 + \frac{1}{2} \int_{\Omega^n} (\nabla \cdot \tilde{\bm{w}}) \tilde{\bm{w}} \cdot \bm \chi_{\bm w} \notag\\
 &= \frac{1}{2} \int_{\Omega^n} \partial_t^\bullet \tilde{\bm{w}} \cdot \bm \chi_{\bm w} + \frac{1}{2} \frac{\rm d}{\rm d t} \bigg|_{t=t_n} \int_{\Omega(t)} \tilde{\bm{w}} \cdot \tilde{\bm \chi}_{\bm w} 
 - \frac{1}{2} \int_{\Omega^n} \tilde{\bm{w}} \cdot \partial_t^{\bullet} \tilde{\bm \chi}_{\bm w},
\end{align}
where $\Omega^n = \Omega(t_n)$ represents the domain at time $t_n$, and $\tilde{\bm \chi}_{\bm w}$ is an extension of $\bm \chi_{\bm w}$ to the evolving domain $\Omega(t)$.

A convenient choice is to construct the extension $\tilde{\bm \chi}_{\bm w}$ using the flow map $\widetilde{\phi}$ from \eqref{flow}. We introduce the pullback mapping
\[
\widetilde{\Psi}^n(t) := \widetilde{\phi}(\cdot,t_n) \circ \widetilde{\phi}(\cdot,t)^{-1} : \Omega(t) \to \Omega^n,
\]
which relates the current domain $\Omega(t)$ to the configuration $\Omega^n$. By defining $\tilde{\bm \chi}_{\bm w} := \bm \chi_{\bm w} \circ \widetilde{\Psi}^n$, the extension satisfies $\partial_t^{\bullet}\tilde{\bm \chi}_{\bm w} = 0$ by construction. Consequently, the identity \eqref{identity-1} simplifies to:
\begin{align}\label{inter0}
    a(\partial_{t}^\bullet \tilde{\bm{w}}, \tilde{\bm{w}}, \bm{\chi}_{\bm w}) = \frac{1}{2} \int_{\Omega^n} \partial_t^{\bullet} \tilde{\bm{w}} \cdot \bm \chi_{\bm w} 
    + \frac{1}{2} \frac{\rm d}{{\rm d}t}\Big|_{t=t_n} \int_{\Omega(t)} \tilde{\bm{w}} \cdot (\bm \chi_{\bm w} \circ \widetilde{\Psi}^n).
\end{align}
This reformulation serves as the basis for the time discretization. Let $\tilde{\bm X}^n := \widetilde{\Psi}^n(t_{n-1}) : \Omega^{n-1} \to \Omega^n$ denote the map from the previous domain $\Omega^{n-1}$ to the current domain $\Omega^n$. 
The first term on the right-hand side of \eqref{inter0} is discretized as:
\begin{align}
\int_{\Omega^n} \partial_t^{\bullet} \tilde{\bm w} \cdot \bm \chi_{\bm w} 
&= \int_{\Omega^n} \frac{\tilde{\bm w}^{n+1} - \tilde{\bm w}^n \circ (\tilde{\bm X}^n)^{-1}}{\tau} \cdot \bm \chi_{\bm w} + \mathcal{O}(\tau) \notag\\
&= \frac{1}{\tau} \int_{\Omega^n} \tilde{\bm w}^{n+1} \cdot \bm \chi_{\bm w} - \frac{1}{\tau} \int_{\Omega^{n-1}} \tilde{\bm w}^n \cdot (\bm \chi_{\bm w} \circ \tilde{\bm X}^n) + \mathcal{O}(\tau), \label{inter1}
\end{align}
where $\tilde{\bm w}^{n+1}$ is defined on $\Omega^n$ and $\tilde{\bm w}^n$ on $\Omega^{n-1}$. The second equality follows from a change of variables, noting that $\det \nabla \tilde{\bm X}^n = 1 + \mathcal{O}(\tau)$ under the first-order approximation of the flow map.
Similarly, the second term on the right-hand side of \eqref{inter0} is approximated by:
\begin{align}
\frac{\rm d}{{\rm d}t}\Big|_{t=t_n} \int_{\Omega(t)} \tilde{\bm{w}} \cdot (\bm \chi_{\bm w} \circ \widetilde{\Psi}^n)
&= \frac{1}{\tau} \left( \int_{\Omega^{n-1}} (\tilde{\bm{w}}^{n+1} \circ \tilde{\bm X}^n) \cdot (\bm \chi_{\bm w} \circ \tilde{\bm X}^n) - \int_{\Omega^{n-1}} 	\tilde{\bm{w}}^n 	\cdot (\bm	\chi_{\bm w}	\circ	\tilde{\bm X}^n)	\right)\notag\\
& \quad\,\, +	\mathcal{O}(\tau).	\label{inter2}
\end{align}

\subsection{Fully discretized linear numerical scheme}\label{subsec:linear_scheme_inertial}

To define the inertial gradient flow in a fully discrete setting, we complement the time discretization introduced in Section~\ref{subsec:transport_discretization_inertial} with a spatial discretization based on the parametric finite element method. We discretize the initial domain $\Omega^0$ by a shape-regular, quasi-uniform triangulation $\mathcal{T}_h^0$ of simplices with mesh size $h$, thus defining the discrete domain $\Omega_h^0$. To represent curved geometries, we employ isoparametric finite elements of degree $k$: each (possibly curved) simplex $K \in \mathcal{T}_h^0$ is given as the image of a polynomial mapping $F_K : \hat K \to K$ of degree $k$ from the reference simplex $\hat K$, which we refer to as the parametrization of $K$.

For each time level $n \ge 1$, the discrete domain $\Omega_h^n$ is defined recursively from $\Omega_h^{n-1}$. Let $\mathcal{T}_h^{n-1}$ denote the triangulation of $\Omega_h^{n-1}$ at time $t_{n-1}$. The finite element space of degree $r$ is defined as
\[
{\mathcal{U}}_{h,r}^{n-1} 
= \left\{ v_h \in C^0(\Omega_h^{n-1}) : v_h|_{K} \in \mathbb{P}^r(K)\ \text{for all } K \in \mathcal{T}_h^{n-1}\right\}.
\]
The domain at time $t_n$ is then defined as the image of $\Omega_h^{n-1}$ under a discrete flow map $\tilde{\bm{X}}_h^n \in ({\mathcal{U}}_{h,r}^{n-1})^d$, i.e.,
\(
\Omega_h^n = \tilde{\bm{X}}_h^n(\Omega_h^{n-1}),
\)
and its boundary $\partial \Omega_h^n$ is denoted by $\Gamma_h^n$, with
\(
\Gamma_h^n = \tilde{\bm{X}}_h^n(\Gamma_h^{n-1}).
\)
The scalar and vector-valued boundary finite element spaces of degree $r$ on $\Gamma_h^{n-1}$ are denoted by $S_{h,r}^{n-1}$ and $(S_{h,r}^{n-1})^d$, respectively.

By incorporating the time discretization introduced in Section~\ref{subsec:transport_discretization_inertial}, we arrive at a computationally efficient linear numerical scheme. Taking the shape reconstruction problem from Section~\ref{sec.ShapeRecon} as a representative example, the fully discrete formulation seeks the updated velocity $\tilde{\bm w}_h^{n+1} \in (\mathcal{U}_{h,r}^n)^d$ together with the state and adjoint variables $(u_h^n, p_h^n) \in \mathcal{U}_{h,r}^n \times \mathcal{U}_{h,r}^n$ such that:
\begin{subequations}\label{descent_scheme}
    \begin{align}
& \frac{\epsilon_0}{2\tau} \int_{\Omega^n_h} \tilde{\bm w}_h^{n+1} \cdot {\bm \chi}_{\bm w}
+ \frac{\epsilon_0}{2\tau} \int_{\Omega_h^{n-1}} (\tilde{\bm w}_h^{n+1} \circ \tilde{\bm X}_h^n) \cdot ({\bm \chi}_{\bm w} \circ \tilde{\bm X}_h^n)  + \eta(t^n) \int_{\Omega^n_h} \left( \tilde{\bm w}_h^{n+1} \cdot {\bm \chi}_{\bm w} + \nabla \tilde{\bm w}_h^{n+1}: \nabla {\bm \chi}_{\bm w} \right) \notag\\
&= \frac{\epsilon_0}{\tau} \int_{\Omega_h^{n-1}} \tilde{\bm w}_h^n \cdot ({\bm \chi}_{\bm w} \circ \tilde{\bm X}_h^n)
- \d J(\Gamma_h^{n},u_h^{n},p_h^{n};\bm \chi_{\bm w}), \label{2ndTimeDisSchemeLinear_0}\\[0.3em]
& \int_{\Omega_h^{n}} (\nabla u_h^{n} \cdot \nabla \chi_u + u_h^{n} \chi_u) 
= \int_{\Omega_h^{n}} f \chi_u, \label{2ndTimeDisSchemeLinear_1}\\[0.3em]
& \int_{\Omega_h^{n}} (\nabla p_h^{n} \cdot \nabla \chi_p + p_h^{n} \chi_p)
= \int_{\Omega_h^{n}} (u_h^{n} - u_d)\chi_p, \label{2ndTimeDisSchemeLinear_2}
\end{align}
\end{subequations}
for all test functions ${\bm \chi}_{\bm w}\in (\mathcal{U}_{h,r}^n)^d$ and $(\chi_u,\chi_p)\in \mathcal{U}_{h,r}^{n} \times \mathcal{U}_{h,r}^n$. 

Once the velocity $\tilde{\bm w}_h^{n+1}$ has been determined, we define the associated flow map
\(
\tilde{\bm X}_h^{n+1} := \mathrm{id} + \tau \tilde{\bm w}_h^{n+1},
\)
which is used to advance the discrete domain from time level $t_n$ to $t_{n+1}$ via
\(
\Omega_h^{n+1} := \tilde{\bm X}_h^{n+1}(\Omega_h^{n}).
\)
In the initial level ($n=0$), the weak formulation \eqref{2ndTimeDisSchemeLinear_0} involves integrals over the auxiliary domain $\Omega_h^{-1}$. To initialize the scheme, we set $\Omega_h^{-1} = \Omega_h^{0}$, $\tilde{\bm w}_h^{0} = \bm 0$, and $\tilde{\bm X}_h^{0} = \mathrm{id}$.

For the shape reconstruction problem \eqref{shapeOptPoisson} and the eigenvalue minimization problem \eqref{EigenPoisson}--\eqref{dJEigen}, we employ isoparametric finite elements, using degree--$1$ spaces for the velocity, state, and adjoint variables. For the Stokes drag minimization problem \eqref{StokesEqu}--\eqref{dJStokes}, equal-order discretizations of the velocity and pressure do not satisfy the inf--sup stability condition, so we use the MINI element pair, which provides a stable approximation for the Stokes equations.
\begin{remark}\upshape
For a fixed time level $n$, the fully discrete scheme~\eqref{descent_scheme} consists of linear subproblems and is therefore linear and uniquely solvable. The state and adjoint problems~\eqref{2ndTimeDisSchemeLinear_1}--\eqref{2ndTimeDisSchemeLinear_2} involve coercive bilinear forms on $\mathcal U_{h,r}^n$ and hence admit unique solutions $(u_h^{n},p_h^{n})$ by the Lax--Milgram lemma. The shape gradient $\d J(\Gamma_h^{n},u_h^{n},p_h^{n};\bm \chi_{\bm w})$ is evaluated using these current state and adjoint variables and therefore acts as a known linear functional in the velocity equation. Consequently, the velocity equation~\eqref{2ndTimeDisSchemeLinear_0} also induces a coercive bilinear form on $(\mathcal U_{h,r}^n)^d$, and the existence and uniqueness of the solution $\tilde{\bm w}_h^{n+1}$ follow directly from the Lax--Milgram lemma.
\end{remark}



\section{Inertial MDR Framework for PDE--Constrained Shape Optimization}
In the continuous setting, the evolution of the domain $\Omega(t)$ is uniquely determined by
the normal component of the boundary velocity on $\Gamma(t)$. 
Tangential motions along the boundary, as well as the extension of the velocity field into the
interior, do not alter the geometric shape of the domain. 
These additional degrees of freedom can therefore be exploited to improve mesh quality during
the evolution, without affecting the underlying physical motion.
By an appropriate choice of the tangential boundary velocity and the interior mesh velocity,
element distortion can be significantly reduced, thereby avoiding frequent remeshing.

The practical task is thus to construct a bulk velocity field in $\Omega(t)$ that propagates
the prescribed normal boundary motion in a mesh--robust manner.
A standard approach is to employ a harmonic extension of the boundary data
\cite{ZhuEig,GongLiRao2024}.
While this strategy is widely used for bulk--surface coupling, purely harmonic extensions are
well known to suffer from severe mesh distortion in the presence of large deformations.

\subsection{Minimal Deformation Rate Framework}\label{subsec:MDR_mesh_motion}

To obtain a mesh motion that better preserves element quality, we adopt the
minimal deformation rate (MDR) principle.
The inertial gradient flow provides a bulk velocity field $\tilde{\bm w}$; however, only its
normal component $\tilde{\bm w}\cdot\bm n$ on the boundary $\Gamma(t)$ is geometrically relevant.
Accordingly, we retain this prescribed normal velocity, while the tangential boundary motion
and the interior extension are determined by minimizing an instantaneous deformation--rate
functional.

With the symmetric velocity gradient $\varepsilon(\cdot)$ defined in \eqref{stress_tensor},
we consider the constrained minimization problem
\begin{equation}\label{defE}
\min_{\bm w\in\mathcal A(t)} \ \mathcal E(\bm w)
:= \frac12 \int_{\Omega(t)} |\varepsilon(\bm w)|^2,
\end{equation}
over the admissible set
\begin{equation}\label{defA}
\mathcal A(t)
:= \Bigl\{ \bm w \in \mathbf H^1(\Omega(t)) :
\ \bm w\cdot\bm n = \tilde{\bm w}\cdot\bm n \ \text{on } \Gamma(t) \Bigr\}.
\end{equation}
We denote by $\bm w^*$ the minimizer of \eqref{defE}--\eqref{defA}, which defines the bulk velocity
field with minimal instantaneous deformation rate among all extensions consistent with the
prescribed normal boundary motion.
To characterize $\bm w^*$, we consider variations
$\delta\bm\phi \in \mathbf H^1(\Omega(t))$ satisfying the homogeneous normal constraint
$\delta\bm\phi\cdot\bm n = 0$ on $\Gamma(t)$.
The first--order optimality condition then reads
\begin{equation}
\int_{\Omega(t)} \varepsilon(\bm w^*) : \varepsilon(\delta\bm\phi) = 0.
\end{equation}
Integrating by parts yields
\begin{equation}\label{MDRiden}
-\int_{\Omega(t)} \operatorname{div}\!\bigl(\varepsilon(\bm w^*)\bigr)\cdot \delta\bm\phi
+ \int_{\Gamma(t)} \bigl(\varepsilon(\bm w^*)\bm n\bigr)\cdot \delta\bm\phi = 0.
\end{equation}
Since $\delta\bm\phi$ is arbitrary in the interior and purely tangential on $\Gamma(t)$,
there exists a scalar Lagrange multiplier $\kappa$ on $\Gamma(t)$ such that the Euler--Lagrange
system can be written as
\begin{subequations}\label{euler-haha}
\begin{align}
\operatorname{div}\!\bigl(\varepsilon(\bm w^*)\bigr) &= 0
\qquad \text{in } \Omega(t), \\
\varepsilon(\bm w^*)\bm n &= \kappa\, \bm n
\qquad \text{on } \Gamma(t),
\end{align}
\end{subequations}
together with the constraint
$\bm w^*\cdot\bm n = \tilde{\bm w}\cdot\bm n$ on $\Gamma(t)$.


At the fully discrete level, the MDR bulk mesh velocity is computed as the discrete analogue
of the variational problem \eqref{MDRiden}--\eqref{euler-haha}.
Specifically, we seek
$({\bm w}_h^{n+1}, \kappa_h^{n+1}) \in (\mathcal{U}_{h,r}^n)^d \times S_{h,r}^n$
such that
\begin{equation}\label{num-bulk-MDR}
\left\{
\begin{aligned}
&\int_{\Omega^{n}_h} \varepsilon({\bm w}_h^{n+1}) : \varepsilon(\bm \eta_h)
= \int_{\Gamma^n_h} \kappa_h^{n+1}\, \bm n_h^n \cdot \bm \eta_h
\quad &&\forall\, \bm \eta_h \in (\mathcal{U}_{h,r}^n)^d,\\
&\int_{\Gamma^n_h} {\bm w}_h^{n+1} \cdot \bm n^{n}_h \, \phi_h
= \int_{\Gamma^n_h} \bigl(\tilde{\bm w}_h^{n+1} \cdot \bm n_h^n\bigr)\, \phi_h
\quad &&\forall\, \phi_h \in S_{h,r}^n.
\end{aligned}
\right.
\end{equation}
Here, $\bm n_h^n$ denotes the piecewise constant unit normal on $\Gamma_h^n$, and
$\tilde{\bm w}_h^{n+1} \cdot \bm n_h^n$ is the discrete normal boundary velocity prescribed
by the inertial numerical scheme \eqref{descent_scheme}.

As a representative example, we consider the shape reconstruction problem introduced in
Section~\ref{sec.ShapeRecon}.
The fully discrete formulation couples the inertial flow with the MDR method
and seeks
\[
(\tilde{\bm w}_h^{n+1}, \bm w_h^{n+1}, \kappa_h^{n+1})
\in (\mathcal{U}_{h,r}^n)^d \times (\mathcal{U}_{h,r}^n)^d \times S_{h,r}^n,
\]
together with the state and adjoint variables
$(u_h^n, p_h^n) \in \mathcal{U}_{h,r}^n \times \mathcal{U}_{h,r}^n$,
such that
\begin{subequations}\label{descent_scheme_MDR}
\begin{align}
& \frac{\epsilon_0}{2\tau} \int_{\Omega^n_h} \tilde{\bm w}_h^{n+1} \cdot {\bm \chi}_{\bm w}
+ \frac{\epsilon_0}{2\tau} \int_{\Omega_h^{n-1}} (\tilde{\bm w}_h^{n+1} \circ \bm X_h^n)
\cdot ({\bm \chi}_{\bm w} \circ \bm X^n_h)
+ \eta(t^n) \int_{\Omega^n_h}
\left( \tilde{\bm w}_h^{n+1} \cdot {\bm \chi}_{\bm w}
+ \nabla \tilde{\bm w}_h^{n+1} : \nabla {\bm \chi}_{\bm w} \right) \notag\\
&\qquad =
\frac{\epsilon_0}{\tau} \int_{\Omega_h^{n-1}} \tilde{\bm w}_h^n
\cdot ({\bm \chi}_{\bm w} \circ \bm X^n_h)
- \d J(\Gamma_h^{n}, u_h^{n}, p_h^{n}; \bm \chi_{\bm w}), \label{2ndTimeDisSchemeLinear_0_MDR}\\[0.3em]
& \int_{\Omega_h^{n}} (\nabla u_h^{n} \cdot \nabla \chi_u + u_h^{n} \chi_u)
= \int_{\Omega_h^{n}} f \chi_u, \label{2ndTimeDisSchemeLinear_1_MDR}\\[0.3em]
& \int_{\Omega_h^{n}} (\nabla p_h^{n} \cdot \nabla \chi_p + p_h^{n} \chi_p)
= \int_{\Omega_h^{n}} (u_h^{n} - u_d) \chi_p, \label{2ndTimeDisSchemeLinear_2_MDR}\\[0.3em]
& \int_{\Omega^{n}_h} \varepsilon({\bm w}_h^{n+1}) : \varepsilon(\bm \eta_h)
= \int_{\Gamma^n_h} \kappa_h^{n+1} \, \bm n_h^n \cdot \bm \eta_h, \label{MDR}\\
& \int_{\Gamma^n_h} ({\bm w}_h^{n+1} \cdot \bm n^{n}_h)\, \phi_h
= \int_{\Gamma^n_h} \tilde{\bm w}_h^{n+1} \cdot \bm n_h^n \, \phi_h, \label{MDR_normal}
\end{align}
\end{subequations}
for all admissible test functions.
It is important to emphasize that the discrete flow map
\[
\bm X_h^{n+1} := \mathrm{id} + \tau \bm w_h^{n+1}
\]
depends explicitly on the mesh velocity $\bm w_h^{n+1}$.
Accordingly, the discrete domain is updated via
\[
\Omega_h^{n+1} := \bm X_h^{n+1}(\Omega_h^{n}).
\]
This differs from the inertial scheme \eqref{descent_scheme},a where the discrete flow map depends only on the inertial velocity $\tilde{\bm w}_h^{n+1}$.
Consequently, the present formulation is genuinely coupled: the flow map $\bm X_h^{n}$ is determined by the mesh velocity $\bm w_h^{n}$ and enters the inertial step \eqref{2ndTimeDisSchemeLinear_0_MDR}, while \eqref{2ndTimeDisSchemeLinear_0_MDR} in turn determines the inertial velocity $\tilde{\bm w}_h^{n+1}$.

At the initial step $n=0$, the weak formulation involves integrals over the auxiliary domain
$\Omega_h^{-1}$.
We initialize the scheme by setting
\[
\Omega_h^{-1} = \Omega_h^{0}, \qquad
\tilde{\bm w}_h^{0} = \bm 0, \qquad
\bm w_h^{0} = \bm 0, \qquad
\bm X_h^{0} = \mathrm{id}.
\]



\subsection{Inertial MDR Framework regularized by surface diffusion}
Initial shapes in optimization may be of low regularity (e.g., sharp corners or non-convex features),
which can impair numerical robustness. A common remedy is area penalization, i.e., augmenting the original
objective by a surface energy term; its $L^{2}$-gradient flow (mean curvature flow) smooths corners
effectively but typically leads to volume loss. By contrast, surface diffusion---the $H^{-1}$-gradient
flow of surface area---regularizes the boundary via local mass redistribution, damping high-curvature
oscillations and reducing geometric irregularities while, for closed surfaces in the continuous
setting, preserving the enclosed volume, a property often desirable in shape optimization.

In our setting, surface diffusion is employed solely to regularize the boundary motion.
We therefore apply a surface--diffusion correction to the velocity produced by the inertial
scheme and subsequently extend the resulting boundary motion into the bulk via the MDR method.
Since the geometric evolution is determined entirely by the normal boundary component, the
diffusion term is incorporated only in the normal direction.
Concretely, starting from the inertial velocity $\tilde{\bm w}$ induced by \eqref{2ndflowExp},
we prescribe in the continuous setting the following composite normal velocity:
\begin{equation}\label{cont_diffusion_descent}
\hat{\bm w} \cdot \bm n
= \tilde{\bm w} \cdot \bm n + \alpha\,(\Delta_{\Gamma} H)
\quad \text{on } \Gamma(t),
\qquad
H = -(\Delta_{\Gamma}\mathrm{id})\cdot \bm n,
\end{equation}
where $H$ denotes the mean curvature of $\Gamma(t)$ and $\mathrm{id}$ is the identity map on $\Gamma(t)$. The parameter $\alpha>0$ controls the relative strength of geometric smoothing versus the primary inertial direction. Throughout, we adopt the adaptive choice
\[
\alpha=\|\tilde{\bm w}\|_{{\bf L}^2(\Gamma(t))},
\]
so that diffusion is strongest at early stages (when $\|\tilde{\bm w}\|$ is large), which helps suppress geometric irregularities of the initial surface, and weakens automatically as the flow approaches a steady regime.

For the spatial discretization of the boundary evolution, we employ the numerical method proposed by Barrett, Garcke, and N\"urnberg (BGN)~\cite{BarrettGarckeNuernberg2007,BarrettGarckeNuernberg2008JCP}. A key feature of this approach is its tangential motion, which induces a discrete (approximate) harmonic map between consecutive surfaces and thereby helps maintain good boundary mesh quality. Specifically, we solve the following variational problem: Find $\hat{\bm w}_h^{n+1}\in (S_{h,r}^n)^d$ and $H_h^{n+1}\in S_{h,r}^n$ such that
\begin{subequations}\label{BGN-surface}
\begin{align}
&\int_{\Gamma_h^n} \hat{\bm w}_h^{n+1}\!\cdot \bm n_h^n \,\psi_h
=\int_{\Gamma_h^n} \tilde{\bm w}_h^{n+1}\!\cdot\bm n_h^n \,\psi_h
-\alpha \int_{\Gamma_h^n} \nabla_{\Gamma_h^n} H_h^{n+1}\cdot \nabla_{\Gamma_h^n}\psi_h
\quad&&\forall\, \psi_h\in S_{h,r}^n, \\
&\int_{\Gamma_h^n} H_h^{n+1}\,\bm \xi_h\!\cdot \bm n_h^n
-\int_{\Gamma_h^n} \nabla_{\Gamma_h^n} ({\rm id}+\tau \hat{\bm w}_h^{n+1})
:\nabla_{\Gamma_h^n} \bm \xi_h = 0
\quad&& \forall\,\bm \xi_h \in (S_{h,r}^n)^d,
\end{align}
\end{subequations}
where $\bm n_h^n$ denotes the piecewise unit normal on the discrete surface $\Gamma_h^n$, and $\tilde{\bm w}_h^{n+1}$ is obtained from the inertial gradient flow scheme \eqref{descent_scheme}.

Once the boundary velocity $\hat{\bm w}_h^{n+1}$ has been determined, the induced motion is extended
into the bulk domain $\Omega_h^n$ using the MDR method introduced in
Section~\ref{subsec:MDR_mesh_motion}, which provides improved mesh quality compared with the
classical harmonic extension. The only difference from the inertial--MDR scheme
\eqref{descent_scheme_MDR} without surface diffusion is that the prescribed normal boundary
velocity in \eqref{MDR_normal} is replaced by $\hat{\bm w}_h^{n+1}\!\cdot\bm n_h^n$ rather than
$\tilde{\bm w}_h^{n+1}\!\cdot\bm n_h^n$. As a result, the bulk velocity is characterized by:
\begin{subequations}\label{descent_scheme_MDR_SD}
    \begin{align}
&\int_{\Omega^{n}_h} \varepsilon({\bm w}_h^{n+1}):\varepsilon(\bm \eta_h)
=\int_{\Gamma^n_h}\kappa_h^{n+1}\,\bm n^n_h\cdot \bm \eta_h,\\
&\int_{\Gamma^n_h} {\bm w}_h^{n+1} \cdot \bm n^{n}_h\,\phi_h
=\int_{\Gamma^n_h} \bigl(\hat{\bm w}^{n+1}_h \cdot \bm n_h^n\bigr)\,\phi_h, \label{MDR_normal_SD}
\end{align}
\end{subequations}
where $\hat{\bm w}_h^{n+1}$ is obtained from the surface--diffusion--regularized boundary
prescription \eqref{BGN-surface}. The resulting inertial diffusion--regularized BGN--MDR algorithm
is summarized in Algorithm~\ref{alg2ndMDR}.

\begin{algorithm}[htbp]
\SetAlgoLined
    \caption{Inertial Diffusion-Regularized BGN--MDR Algorithm}\label{alg2ndMDR}
    \KwData{Given the terminal time $T$ and time step $\tau$, initialize the domain $\Omega^0_h$ and set $\Omega_h^{-1} = \Omega_h^{0}, 
\qquad \tilde{\bm w}_h^{0} = \bm 0, 
\qquad \bm{w}_h^{0} = \bm 0, 
\qquad \text{and} \qquad 
{\bm X}_h^{0} = \mathrm{id}$.\\ \qquad \quad \ Set the current time $t=0$ and index $n=0$.}
    \While{$t \leq T$}{
     \emph{Step 1}: Solve the governing PDE on the current domain $\Omega_h^n$, see \eqref{2ndTimeDisSchemeLinear_1_MDR}--\eqref{2ndTimeDisSchemeLinear_2_MDR}.\\ 
     \emph{Step 2}: Compute $\tilde{\bm w}^{n+1}_h$ from inertial gradient flow scheme~\eqref{2ndTimeDisSchemeLinear_0_MDR}.\\
     \emph{Step 3}: Compute the surface--diffusion--regularized boundary velocity $\hat{\bm w}^{n+1}_h$ on $\Gamma_h^n$\\
	 \hspace{35pt} by using the BGN method~\eqref{BGN-surface}.\\
	 \emph{Step 4}: Compute the domain velocity ${\bm w}^{n+1}_h$ on $\Omega_h^n$ using the MDR method~\eqref{descent_scheme_MDR_SD}.\\
     \emph{Step 5}: Update the domain by $\Omega_h^{n+1} = \bm {X}_h^{n+1}(\Omega_h^n)$ with $\bm{{X}}^{n+1}_h:={\rm id}+ \tau {\bm w}_h^{n+1}$.\\
     Set $t \leftarrow t+\tau$ and $n\leftarrow n+1$.}
     \label{alg1}
\end{algorithm}

\section{Inertial MDR method for Willmore-driven surface hole filling}

In this section, we develop a numerical scheme for the Willmore-driven surface hole filling problem, modeled by the continuous flow \eqref{willmore-flow} subject to the boundary conditions \eqref{boundary}.

The first boundary condition \eqref{boundary_1} requires the boundary curve to remain fixed throughout the evolution, i.e., $\partial \Gamma(t) = \partial \Gamma(0)$ for all $t$. Let \(\mathring{S}_{h,r}^n\) denote the subspace of \(S_{h,r}^n\) comprising functions that vanish on the boundary \(\partial \Gamma_h^n\). By seeking the discrete velocity $\Bw_{h}^{n+1}$ within $(\mathring{S}_{h,r}^n)^3$, we explicitly satisfy the condition that no displacement occurs on the boundary, thereby fulfilling \eqref{boundary_1}.

The second boundary condition \eqref{boundary_2}, which prescribes the conormal direction, will be enforced weakly. We begin with the following geometric identity, valid for smooth surfaces (derived via integration by parts and the relation $H = -(\Delta_{\Gamma} {\rm\text{id}}) \cdot \bm n$):
\begin{align}\label{willmore-identity}
    \int_{\Gamma(t)} H \chi 
    = \int_{\Gamma(t)} \nabla_{\Gamma(t)} {\rm \text{id}} : \nabla_{\Gamma(t)} (\chi \bm n)  
    + \int_{\partial \Gamma(t)} {\bm \mu}(t) \cdot \bm n(t)\,\chi,
\end{align}
for all $\chi \in H^1(\Gamma(t))$. Here, $\bm\mu(t)$ denotes the conormal vector along $\partial\Gamma(t)$. Although the boundary term vanishes in the continuous setting due to the orthogonality condition $\bm\mu(t) \perp \bm n(t)$, we explicitly retain it to impose the boundary constraint.

Specifically, we propose the following variational formulation for mean curvature:
\begin{align}\label{willmore-identity-2}
    \int_{\Gamma(t)} H \chi 
    = \int_{\Gamma(t)} \nabla_{\Gamma(t)} {\rm id} : \nabla_{\Gamma(t)} (\chi \bm n)  
    + \int_{\partial \Gamma(0)} {\bm \mu}(0) \cdot \bm n(t)\,\chi.
\end{align}
Comparing the identity \eqref{willmore-identity} with our formulation \eqref{willmore-identity-2}, and utilizing the fact that $\partial \Gamma(t) = \partial \Gamma(0)$, we deduce that enforcing \eqref{willmore-identity-2} implies
\begin{align}
    \int_{\partial \Gamma(0)} {\bm \mu}(t) \cdot \bm n(t)\,\chi 
    = \int_{\partial \Gamma(0)} {\bm \mu}(0) \cdot \bm n(t)\,\chi.
\end{align}
Since this equality holds for arbitrary test functions, we have ${\bm \mu}(t) \cdot \bm n(t) = {\bm \mu}(0) \cdot \bm n(t)$ almost everywhere on the boundary. Furthermore, as the boundary curve is fixed, both ${\bm \mu}(t)$ and ${\bm \mu}(0)$ are orthogonal to the boundary tangent $\bm \xi$. Consequently, this implies that ${\bm \mu}(t) = {\bm \mu}(0)$. Provided that the initial data is compatible (i.e., ${\bm \mu}(0) = \bm \mu_{\rm out}$), the second boundary condition \eqref{boundary_2} is satisfied.

A specific numerical challenge arises in the discretization of \eqref{willmore-identity-2}, as the discrete unit normal $\bm n_h^n$ is typically discontinuous across element interfaces and thus lacks the global $H^1$-regularity required to define the surface gradient. To address this issue, we introduce an averaged normal field $\bar{\bm n}_h^n \in (S_{h,r}^n)^3$, defined as the $L^2$-projection of $\bm n_h^n$ onto the finite element space:
\begin{align}\label{normal}
    \int_{\Gamma_h^n} \bar{\bm n}_h^n \cdot \bm \chi_{\bm n} 
    = \int_{\Gamma_h^n} \bm n_h^n \cdot \bm\chi_{\bm n} 
    \quad \text{for all } \bm \chi_{\bm n} \in (S_{h,r}^n)^3.
\end{align}
By construction, $\bar{\bm n}_h^n$ belongs to $\mathbf{H}^1(\Gamma_h^n)$, which ensures that the
subsequent weak formulation is well defined.
Moreover, this projection-based treatment improves numerical robustness by mitigating
the long-time accumulation of geometric errors, without introducing additional evolution
equations for geometric quantities.


Building upon the inertial--MDR framework for PDE--constrained shape optimization developed earlier,
we develop a linear fully discrete scheme for Willmore flow on open surfaces with a fixed boundary.
The method incorporates a second--order inertial term to accelerate the evolution and couples it
with the minimal deformation rate (MDR) principle to maintain mesh quality without remeshing.

Since the geometric evolution of a surface is governed entirely by its normal velocity,
the tangential component does not affect the shape.
Accordingly, the normal velocity for the Willmore flow satisfies
\begin{align}\label{willmore-flow-normal}
\bm w \cdot \bm n_{\Gamma(t)}
= \Big(\Delta_{\Gamma(t)} H_{\Gamma(t)} 
- \tfrac12\,H_{\Gamma(t)}\big(H_{\Gamma(t)}^{2}
-2|\nabla_{\Gamma(t)} \bm n_{\Gamma(t)}|^2\big)\Big)
\quad \text{on } \Gamma(t).
\end{align}
To incorporate inertial effects, we replace the above first--order normal--velocity relation
with its inertial counterpart:
\begin{align}\label{willmore-flow-normal-inertial}
\big[\epsilon_0 \big(\partial_t^{\bullet} \bm w 
+ \tfrac12 (\nabla_{\Gamma(t)}\cdot \bm w)\bm w\big) + \bm w \big]
\cdot \bm n_{\Gamma(t)}
= \Big(\Delta_{\Gamma(t)} H_{\Gamma(t)} 
- \tfrac12\,H_{\Gamma(t)}\big(H_{\Gamma(t)}^{2}
-2|\nabla_{\Gamma(t)} \bm n_{\Gamma(t)}|^2\big)\Big).
\end{align}

To preserve mesh quality during the evolution, the tangential component of the velocity
is determined by minimizing the deformation--rate energy
\[
\int_{\Gamma(t)} |\varepsilon(\bm w)|^2,
\qquad
\varepsilon(\bm w):=\tfrac12\bigl(\nabla_{\Gamma(t)} \bm w
+(\nabla_{\Gamma(t)} \bm w)^{\top}\bigr),
\]
subject to the normal constraint \eqref{willmore-flow-normal-inertial}.
The resulting Euler--Lagrange equation is analogous to \eqref{euler-haha} and leads to the
variational condition \eqref{tangent-vel} in the fully discrete scheme.

The fully discrete linear scheme is defined as follows: Find
\[
(\bm w_h^{n+1}, \kappa_h^{n+1}, H_h^{n+1})
\in (\mathring{S}_{h,r}^n)^3 \times \mathring{S}_{h,r}^n \times S_{h,r}^n
\]
such that
\begin{subequations}\label{numer-willmore}
\begin{align}
&\frac{\bm X_h^{n+1} - \mathrm{id}}{\tau} = \bm w_h^{n+1} \qquad \text{on } \Gamma_h^n,\label{position}\\
&\epsilon_0\int_{\Gamma_h^n} \frac{\bm w_h^{n+1}\cdot \bar{\bm n}_h^n}{2\tau} \chi_v 
+ \epsilon_0\int_{\Gamma_h^{n-1}} \frac{(\bm w_h^{n+1}\cdot \bar{\bm n}_h^n) \circ \bm X^n_h }{2\tau} (\chi_v \circ \bm X^n_h) 
+ \int_{\Gamma_h^n}(\bm w_h^{n+1} \cdot \bar{\bm n}_h^n) \chi_v \notag\\
&\quad + \int_{\Gamma_h^n}\nabla_{\Gamma_h^n} H_h^{n+1}\cdot \nabla_{\Gamma_h^n} \chi_v 
= \epsilon_0\int_{\Gamma_h^{n-1}} \frac{\bm w_h^{n}\cdot \bar{\bm n}_h^{n-1}}{\tau} (\chi_v \circ \bm X^n_h)\notag\\
& \quad + \int_{\Gamma_h^n} \left((H_h^n)^3 - \frac{3}{2} (H_h^n)^2 H_h^{n+1} 
+|\nabla_{\Gamma_h^n} \bar{\bm n}_h^n|^2 H_h^{n+1}\right) \chi_v, \label{velocity-eqn} \\
&\int_{\Gamma_h^n} \varepsilon^n (\bm{w}_h^{n+1}): \varepsilon^n (\bm{\chi}_\kappa) 
= \int_{\Gamma_h^n} \kappa_h^{n+1} \bar{\bm n}_h^n \cdot \bm{\chi}_\kappa, \label{tangent-vel} \\
&\int_{\Gamma_h^n} H_h^{n+1} \chi_H 
= \int_{\Gamma_h^n} \nabla_{\Gamma_h^n} (\mathrm{id} + \tau \bm w_h^{n+1}): \nabla_{\Gamma_h^n} (\chi_H \bar{\bm n}_h^n) 
+ \int_{\partial \Gamma_h^0} \bm{\mu}_{\rm out} \cdot \bar{\bm n}_h^n \, \chi_H, \label{curvature-eqn}
\end{align}
\end{subequations}
for all test functions \((\bm \chi_\kappa, \chi_v, \chi_H) \in (\mathring{S}_{h,r}^n)^3 \times \mathring{S}_{h,r}^n\times {S}_{h,r}^n\). In \eqref{tangent-vel}, the strain tensor is defined as $\varepsilon^n(\bm v):=\tfrac12\bigl(\nabla_{\Gamma_h^n} \bm v +(\nabla_{\Gamma_h^n} \bm v)^{\top}\bigr)$.
Upon solving the above linear system, the discrete surface is advanced by the discrete flow map
\(\bm X_h^{n+1}\), defined through
\[
\Gamma_h^{n+1} := \bm X_h^{n+1}(\Gamma_h^n)
= (\mathrm{id} + \tau \bm w_h^{n+1})(\Gamma_h^n),
\]
with the convention \(\Gamma_h^{-1} := \Gamma_h^0\).
To initialize the scheme, we prescribe \(\bm w_h^0 = \bm 0\) and \(\bm X_h^0 = \mathrm{id}\).
The discrete mean curvature \(H_h^0\) on the initial surface \(\Gamma_h^0\) is then obtained from the weak formulation
\begin{align}\label{known_curvature}
    \int_{\Gamma_h^0} H_h^0 \psi_H
    = \int_{\Gamma_h^0} \nabla_{\Gamma_h^0} \cdot \bar{\bm n}_{h}^0 \,\psi_H,
    \quad \forall\, \psi_H \in S_{h,r}^0.
\end{align}

Equation~\eqref{velocity-eqn} governs the evolution of the normal velocity, balancing inertial
effects (scaled by~$\epsilon_0$) with the linearized Willmore driving force. The tangential
component is selected implicitly through the MDR principle~\eqref{tangent-vel} to preserve mesh
quality. Finally,~\eqref{curvature-eqn} closes the system by updating the mean curvature and
simultaneously enforcing the prescribed conormal boundary condition~\eqref{boundary_2} in a weak
sense, which is the discrete analogue of~\eqref{willmore-identity-2}. 
In practice, the evolution is terminated once the objective functional becomes numerically
stationary, that is, when its relative variation over successive time steps falls below a
prescribed tolerance, indicating convergence to a steady configuration.

\begin{remark}\upshape
The constraint \(\bm w_h^{n+1}\in (\mathring{S}_{h,r}^n)^3\) enforces the fixed-boundary condition
\eqref{boundary_1}, while the conormal condition \eqref{boundary_2} is imposed weakly through \eqref{curvature-eqn}. This weak enforcement is robust with respect to
incompatible initial data (i.e., \(\bm \mu(0)\neq \bm \mu_{\rm out}\)), allowing the discrete
conormal to relax toward the target \(\bm \mu_{\rm out}\) during the evolution. This feature is
particularly useful for surface hole filling, as it permits simple initial guesses \(\Gamma(0)\)
that need not satisfy the geometric boundary compatibility.
\end{remark}


\section{Numerical experiment}\label{sec:numerics}
The PDE-constrained shape optimization problems are solved numerically using the inertial BGN--MDR scheme described in Algorithm~\ref{alg1}. We consider three benchmark shape optimization problems: Example~1 (shape reconstruction), Example~2 (drag minimization in Stokes flow), and Example~3 (elliptic eigenvalue minimization). In addition, Example~4 (surface hole filling) is included to illustrate the performance of the proposed inertial--MDR scheme for Willmore-flow-driven surface reconstruction. Taken together, these experiments demonstrate that the inertial--MDR framework achieves lower
original objective values \(J(\Gamma)=\int_{\Omega} j(x,u)\), converges more rapidly, and produces higher--quality optimized geometries with
improved mesh quality compared with classical numerical methods.

\subsection{Example 1 (Shape reconstruction)}
We investigate the shape reconstruction problem~\eqref{shapeOptPoisson} in both two and three spatial dimensions, considering initial geometries that are either smooth or non-smooth, as well as convex or non-convex.

\noindent
\textbf{Case 1 (Inertial gradient flow).}
The desired state is prescribed as
\[
u_d(x_1,x_2) = 1 - 0.8x_1^2 - 2.25x_2^2.
\]
The initial domain is a unit disk centered at the origin, discretized into 2{,}842 triangular elements. The time step size and final time are chosen as \(\tau = 0.02\) and \(T = 8\), respectively. As shown in Fig.~\ref{Exp2case2}, the second--order (inertial) gradient flow with \(\epsilon_0=1\) provides a noticeably better approximation of the target elliptical shape than the classical \(H^1\) gradient flow. The convergence histories (Fig.~\ref{Exp2case2}, right) show that the inertial flow converges rapidly, reducing the original objective value to \(10^{-7}\). In contrast, the \(H^1\) gradient flow converges much more slowly and stagnates on a plateau, with the objective value only decreasing to \(10^{-3}\), which is insufficient for accurately recovering the target shape. Moreover, for the inertial flow, the mechanical energy defined in \eqref{modified-energy} decays monotonically, as shown in Fig.~\ref{Exp2case2}, right.

\noindent
\textbf{Case 2 (Inertial BGN--MDR method in 2D).}
The initial domain is a non-convex and non-smooth L-shaped region
\[
\Omega^0 = [-1,1]^2 \setminus \bigl((0,1)\times(0,1)\bigr).
\]
discretized into 1{,}484 triangular elements. The desired state is given by
\[
u_d(x_1,x_2) = 1 - x_1^2 - 2.25x_2^2.
\]
All remaining parameters are identical to those in Case~1. The inertial BGN--MDR method yields a high-quality reconstructed shape and improves the mesh quality compared with the BGN--harmonic extension, in particular by preventing severe element distortion near the re-entrant corners; see Fig.~\ref{Exp1Case3MDRIniOpt}. 

\noindent
\textbf{Case 3 (Inertial BGN--MDR method in 3D).}
The initial domain is the cuboid
\[
\Omega^0 = [-0.75,0.75]\times[-0.5,0.5]\times[-0.5,0.5].
\]
We set \(\tau = 0.1\), \(\epsilon_0 = 1\), and \(T = 10\). The desired state is prescribed as
\[
u_d(x_1,x_2,x_3) = 1 - x_1^2 - 2.25x_2^2 - 2.25x_3^2.
\]
As illustrated in Fig.~\ref{Exp23d}, the optimized shape obtained by the inertial BGN--MDR method (regularized by surface diffusion) exhibits a substantially smoother surface than that produced by the classical \(H^1\) gradient flow or inertial MDR method. Quantitatively, the original objective value achieved by the inertial BGN--MDR method is \(6.8\times10^{-6}\), which is substantially smaller than the value \(1.4\times10^{-5}\) attained by the inertial MDR flow and the value \(1.7\times 10^{-3}\) attained by the classic $H^1$ gradient flow.

\begin{figure}[htbp]
    \centering
    \begin{minipage}[t]{0.33\textwidth}
        \centering
        \includegraphics[width=1.75in]{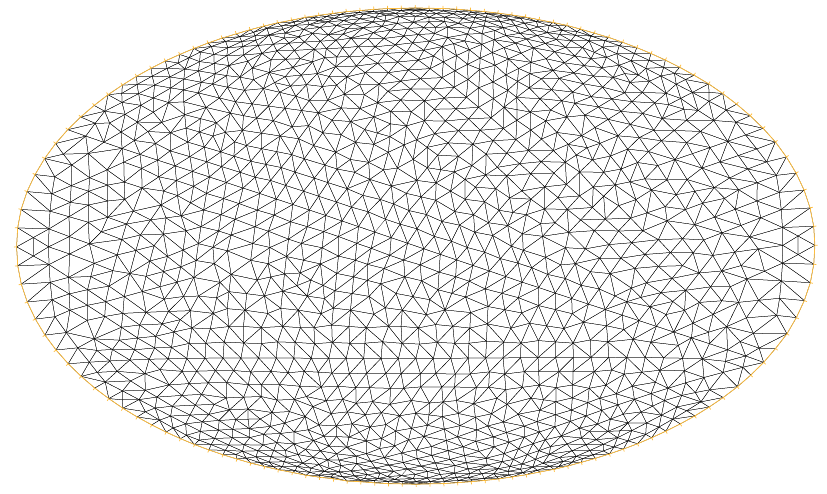}
    \end{minipage}%
    \begin{minipage}[t]{0.33\textwidth}
        \centering
        \includegraphics[width=1.75in]{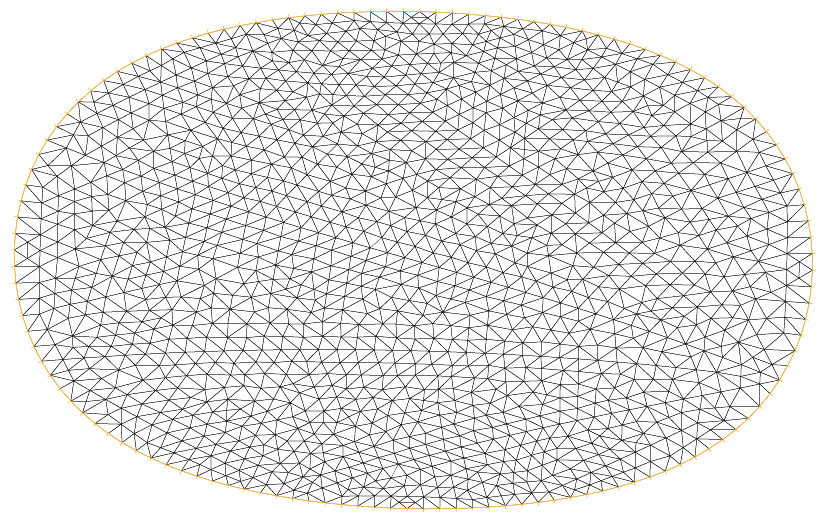}
    \end{minipage}%
    \begin{minipage}[t]{0.33\textwidth}
        \centering
        \includegraphics[width=2.0in]{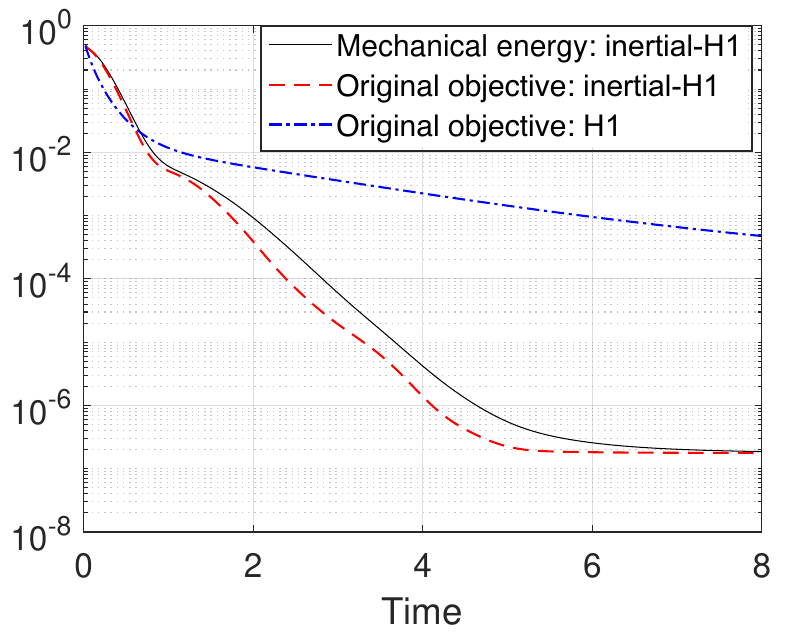}
    \end{minipage}%
    \caption{Optimal shapes produced by the inertial \(H^1\) gradient flow (left) and the classical \(H^1\) gradient flow (middle), together with the corresponding original objective and mechanical energy convergence histories (right), for Example~1 (Case~1).}
    \label{Exp2case2}
\end{figure}

\begin{figure}[htbp]
    \centering
    \begin{minipage}[t]{0.24\textwidth}
        \centering
        \includegraphics[width=1.2in]{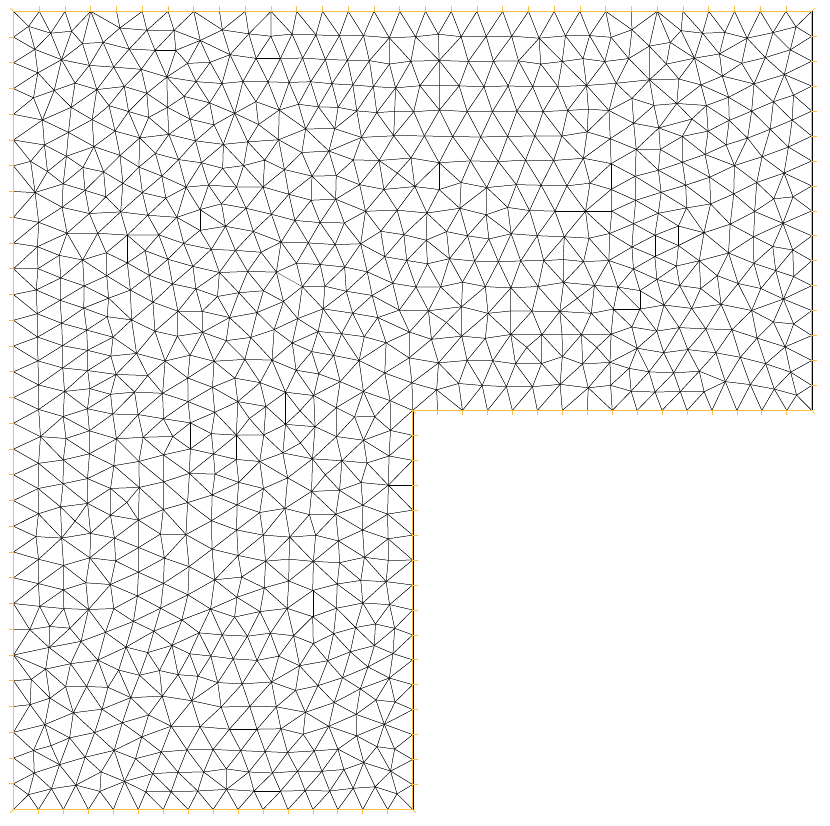}
    \end{minipage}%
    \begin{minipage}[t]{0.24\textwidth}
        \centering
        \includegraphics[width=1.3in]{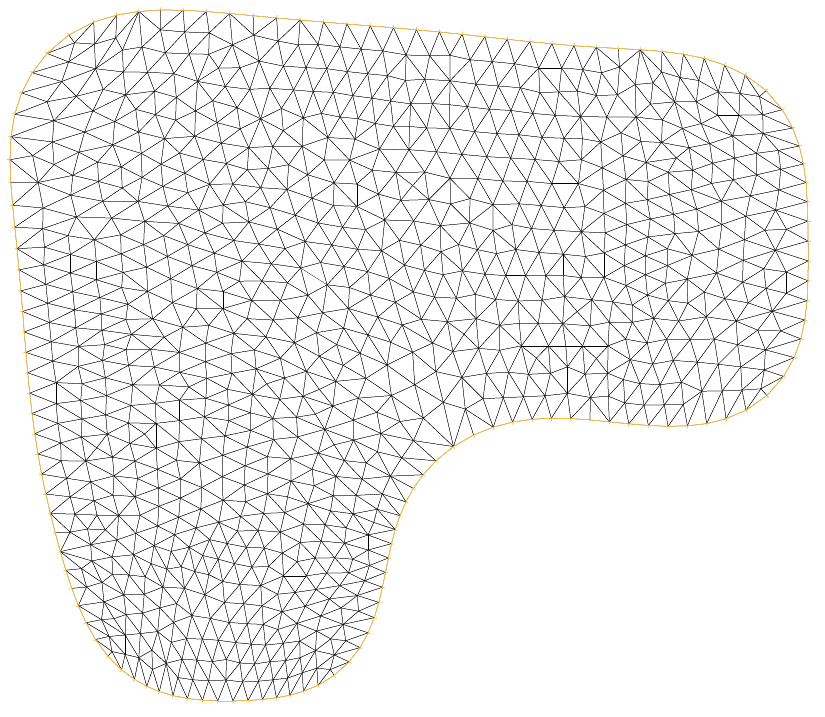}
    \end{minipage}%
    \begin{minipage}[t]{0.24\textwidth}
        \centering
        \includegraphics[width=1.5in]{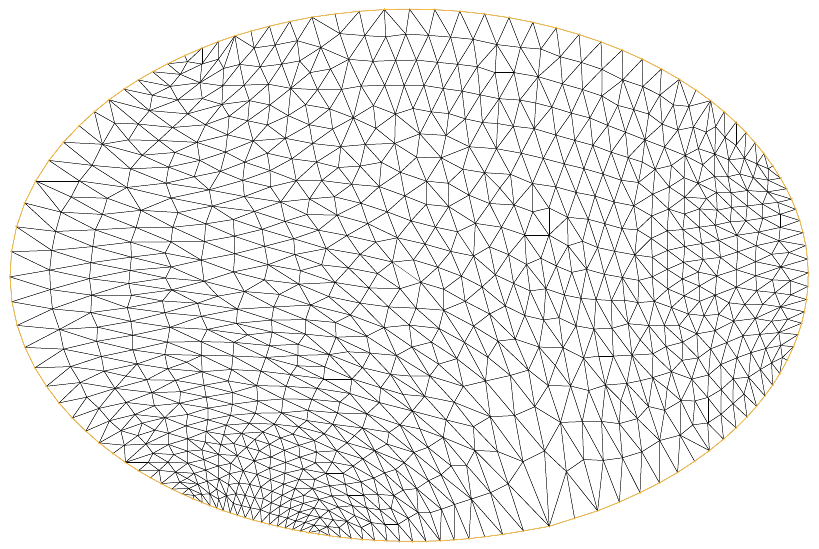}
    \end{minipage}%
    \begin{minipage}[t]{0.27\textwidth}
        \centering
        \includegraphics[width=1.5in]{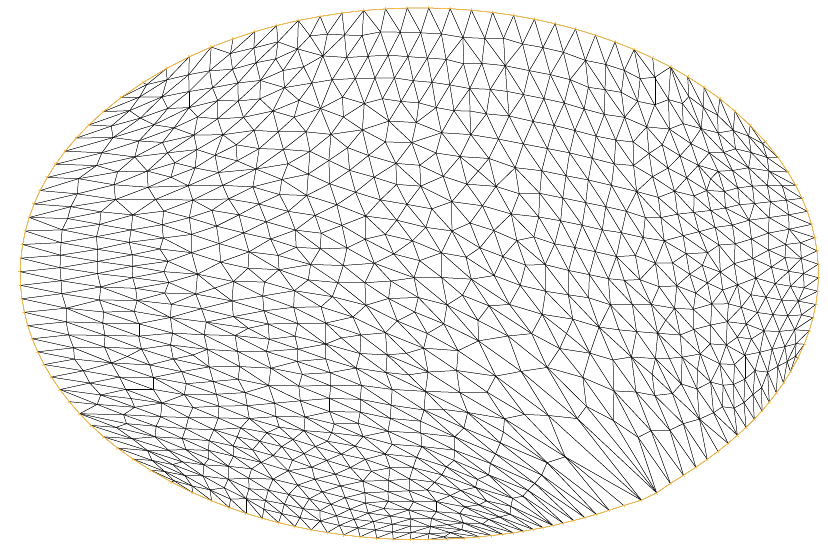}
    \end{minipage}%
    \caption{Mesh evolution for Example~1 (Case~2): initial mesh (left), intermediate stage (middle left), final optimized mesh via the inertial BGN--MDR flow (middle right), and the BGN--harmonic extension (right).}
    \label{Exp1Case3MDRIniOpt}
\end{figure}

\begin{figure}[htbp]
    \centering
    \begin{minipage}[t]{0.33\textwidth}
        \centering
        \includegraphics[width=1.8in]{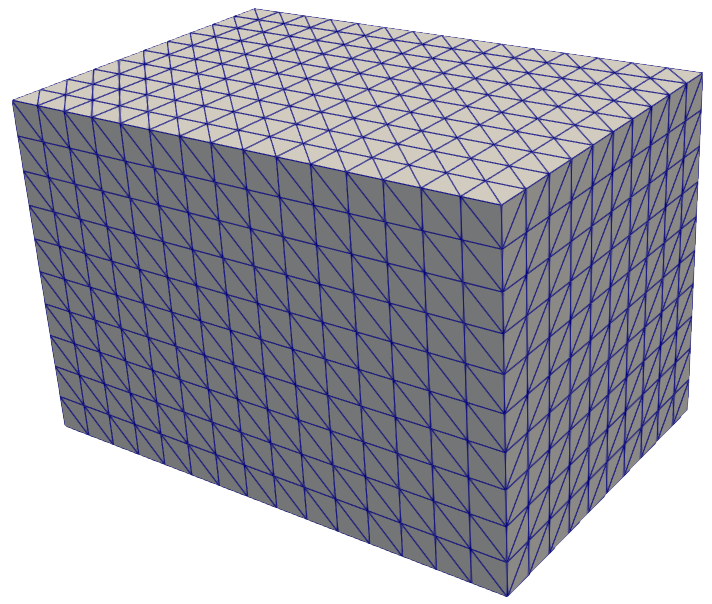}
    \end{minipage}%
    \begin{minipage}[t]{0.33\textwidth}
        \centering
        \includegraphics[width=1.8in]{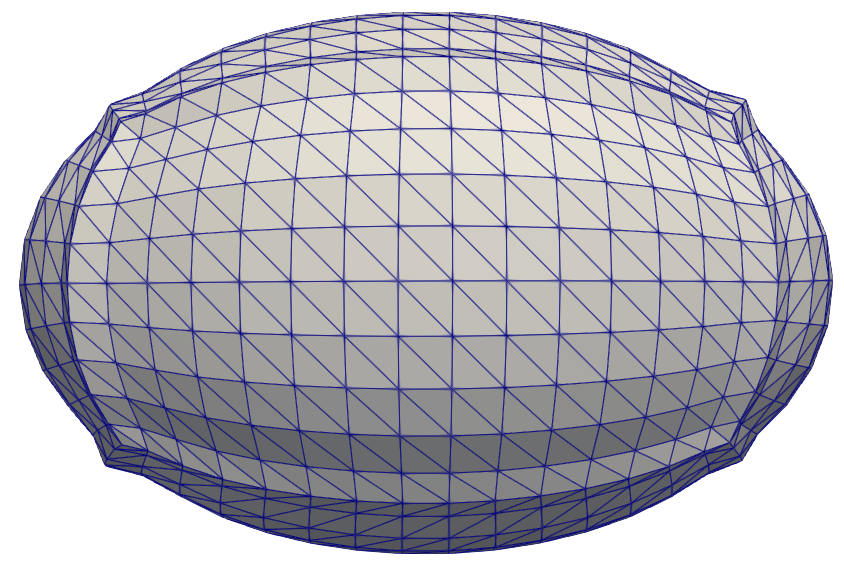}
    \end{minipage}%
    \begin{minipage}[t]{0.33\textwidth}
        \centering
        \includegraphics[width=1.8in]{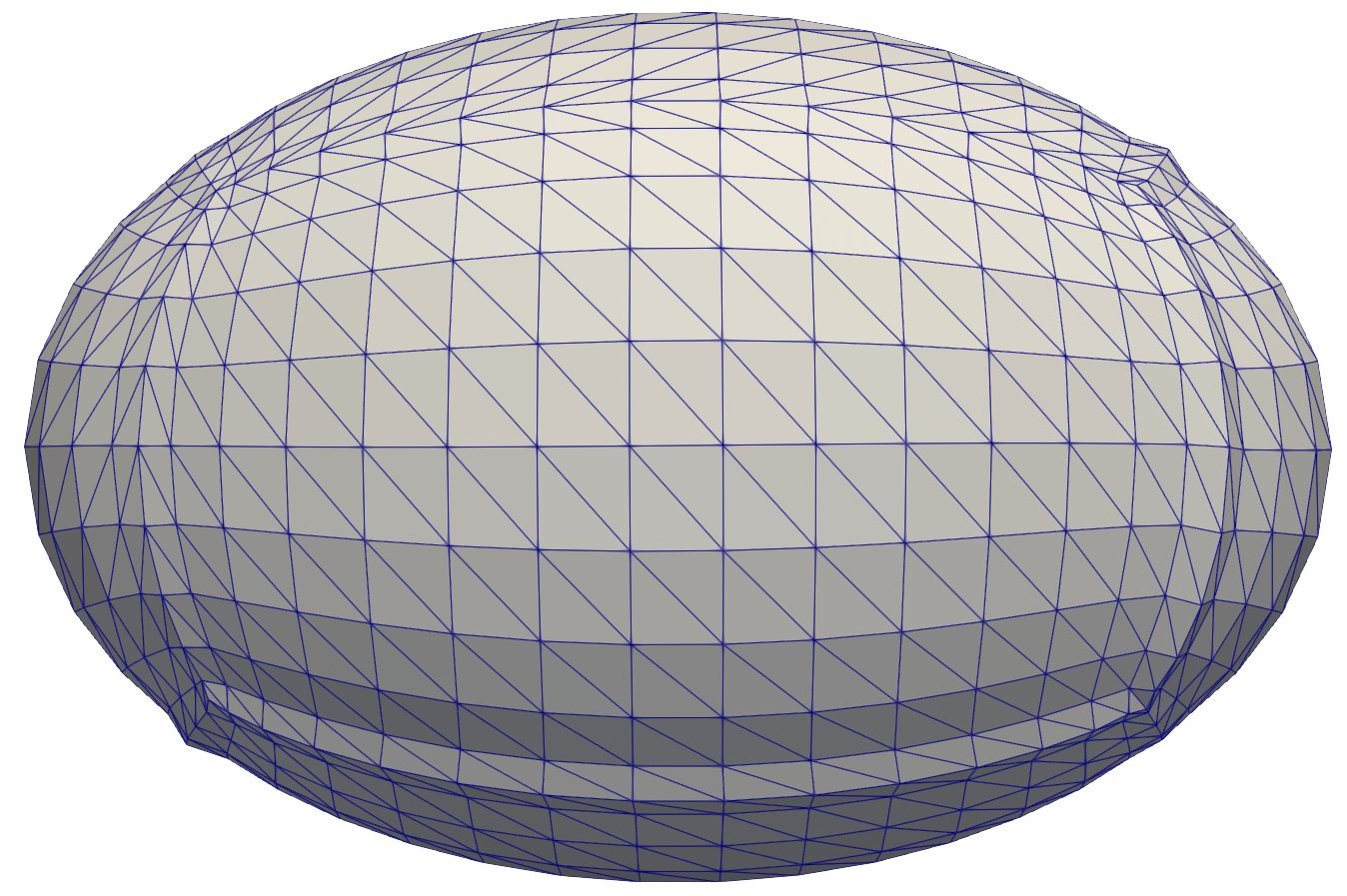}
    \end{minipage}%
    \\
    \begin{minipage}[t]{0.5\textwidth}
        \centering
        \includegraphics[width=1.8in]{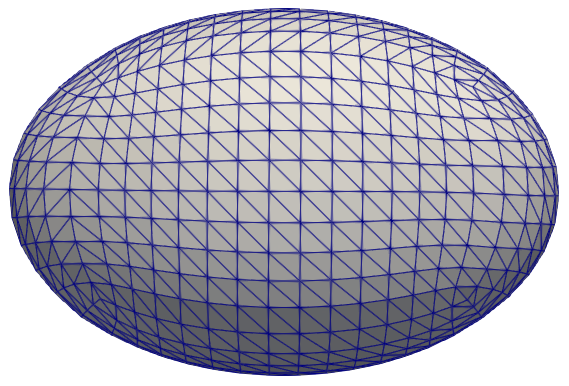}
    \end{minipage}%
    \begin{minipage}[t]{0.5\textwidth}
        \centering
        \includegraphics[width=1.8in]{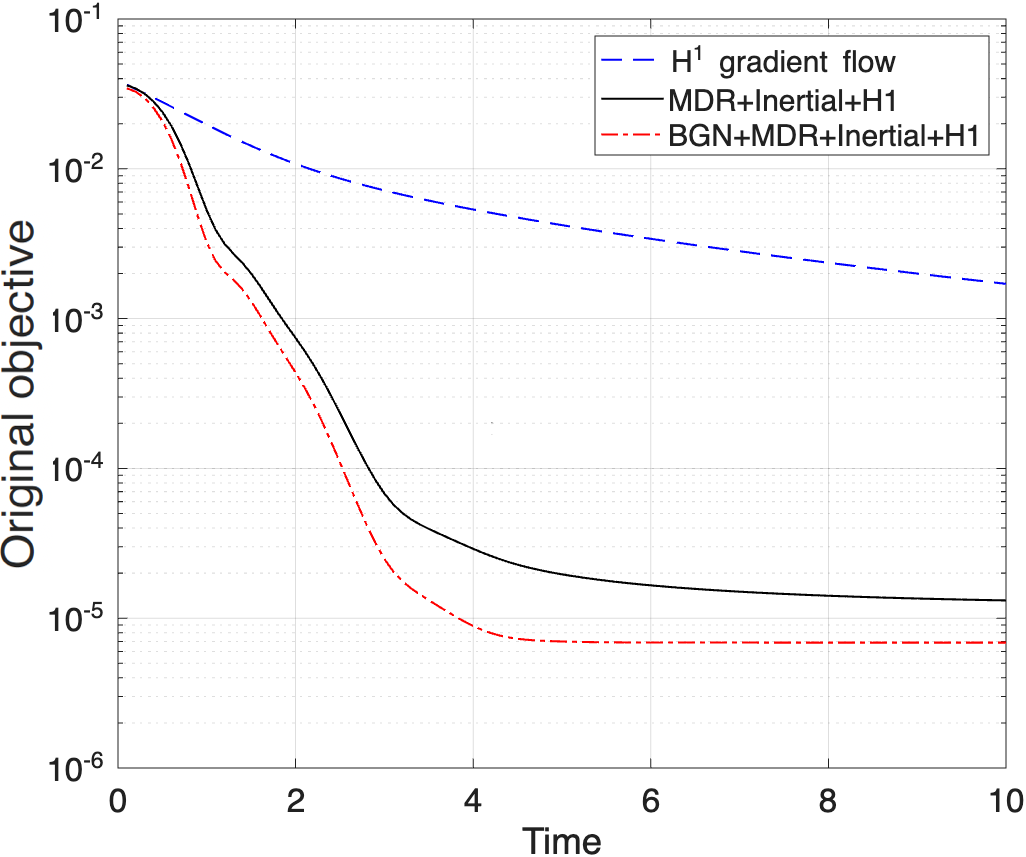}
    \end{minipage}%
    \caption{Example~1 (Case~3): Row 1 shows initial mesh (left), optimized mesh obtained by the \(H^1\) gradient flow (middle), and optimized mesh obtained by the inertial MDR method (right); Row 2 shows optimized mesh obtained by inertial BGN--MDR method (left) and convergence histories of original objective (right).}
    \label{Exp23d}
\end{figure}

\subsection{Example 2 (Drag minimization in Stokes flow)}
The optimal obstacle for drag minimization typically develops sharp tips at the leading and trailing edges, while the remaining part of the boundary stays smooth. In \eqref{StokesEqu}, we set the body force to \(\bm f \equiv \bm 0\).

\noindent
\textbf{Case 1 (Non-smooth initial shape).}
The initial design consists of a rectangular domain \(\Omega^0=[-0.5,1.5]\times[-0.5,0.5]\) containing a rectangular obstacle \([-0.2,0.2]\times[-0.15,0.15]\) removed from the fluid region; see Fig.~\ref{Stokes2dDrag} (left). The inflow profile on \(\Gamma_i\) is prescribed as
\[
\bm u_d = \bigl(-(x_2-0.5)(x_2+0.5),\,0\bigr)^{\top}.
\]
We choose \(\tau=0.005\), \(T=0.2\), and \(\mu=1\). Figure~\ref{Stokes2dDrag} shows the shape evolution computed by the inertial BGN--MDR method (regularized by surface diffusion). While the sharp corners of the initial rectangular obstacle are smoothed out at intermediate times, the optimized configuration develops pronounced tips at the leading and trailing edges, which promotes a streamlined flow pattern; see Fig.~\ref{StokesVelObj2d} (left). As reported in Fig.~\ref{StokesVelObj2d} (middle), the inertial BGN--MDR method achieves a lower original objective value than the classical \(H^1\) gradient flow. Moreover, Fig.~\ref{StokesVelObj2d} (right) shows the time history of the mechanical energy, which
exhibits a stable decay in our experiments.


\noindent
\textbf{Case 2 (Obstacle design in 3D).}
The initial obstacle is an ellipsoid placed inside the cuboidal design domain \([0,1.5]\times[0,1]\times[0,1]\); see Fig.~\ref{StokesDrag3d} (left). We apply the inertial BGN--MDR method (regularized by surface diffusion) with \(\tau=0.001\), \(T=0.05\), \(\epsilon_0=0.1\), and a constant inflow \(\bm u_d=(0.5,0,0)^{\top}\). The resulting optimized obstacle exhibits sharp tips at the leading and trailing edges while the rest of the surface remains smooth; see Fig.~\ref{StokesDrag3d} (right). Throughout the computation, the mesh quality is well maintained, illustrating the effectiveness of the inertial BGN--MDR method for three-dimensional drag minimization.

\begin{figure}[htbp]
    \centering
    \begin{minipage}[t]{0.33\textwidth}
        \centering
        \includegraphics[width=1.95in]{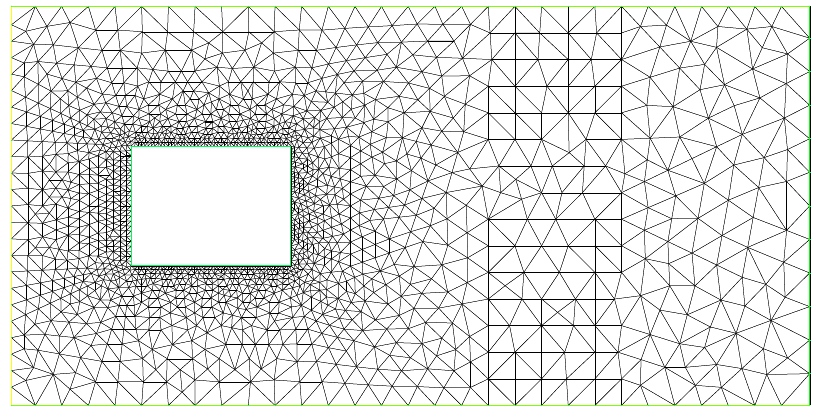}
    \end{minipage}%
    \begin{minipage}[t]{0.33\textwidth}
        \centering
        \includegraphics[width=1.95in]{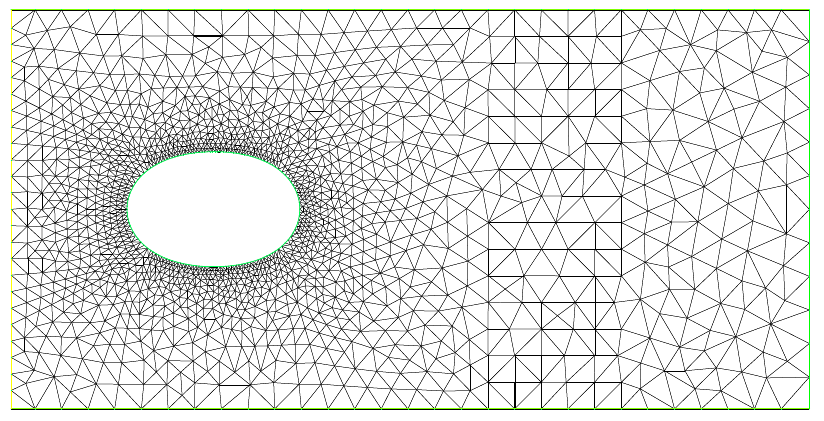}
    \end{minipage}%
    \begin{minipage}[t]{0.33\textwidth}
        \centering
        \includegraphics[width=1.95in]{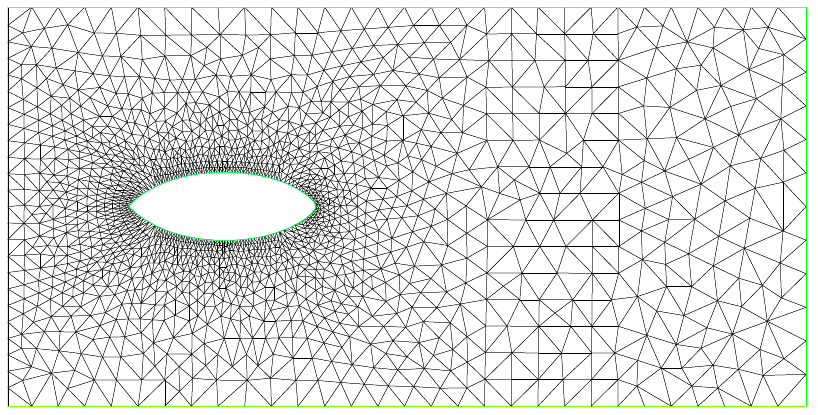}
    \end{minipage}%
    \caption{Example~2 (Case~1): initial mesh (left), intermediate mesh (middle), and optimized mesh (right) obtained by the inertial BGN--MDR method.}
    \label{Stokes2dDrag}
\end{figure}

\begin{figure}[htbp]
    \centering
    \begin{minipage}[t]{0.33\textwidth}
        \centering
        \includegraphics[width=2.2in]{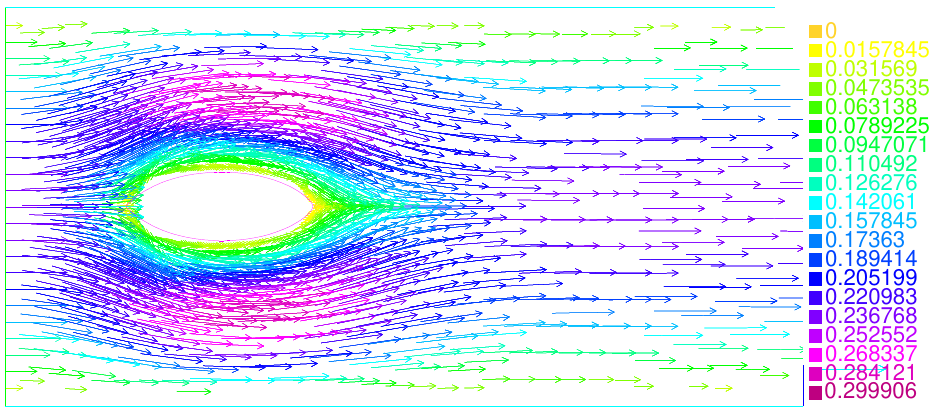}
    \end{minipage}%
    \begin{minipage}[t]{0.33\textwidth}
        \centering
        \includegraphics[width=1.8in]{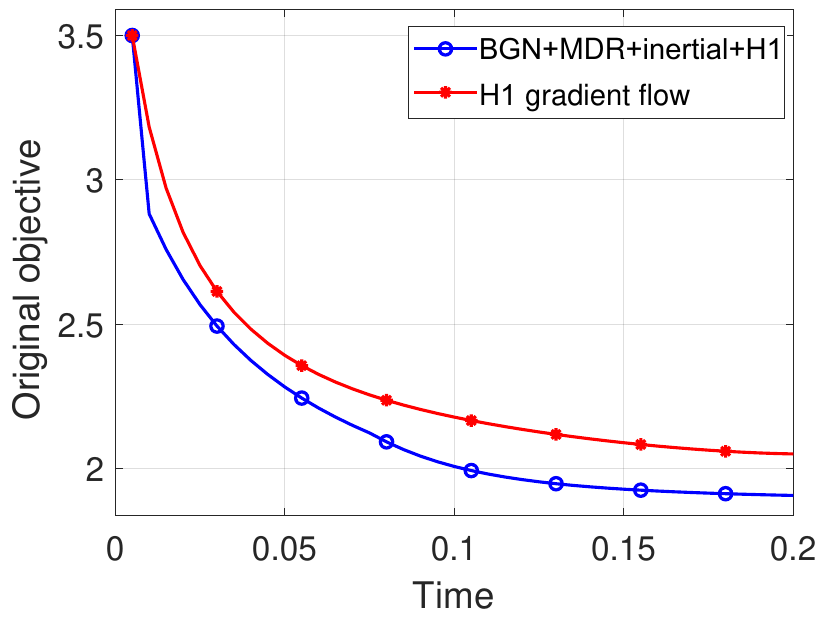}
    \end{minipage}%
    \begin{minipage}[t]{0.33\textwidth}
        \centering
        \includegraphics[width=1.8in]{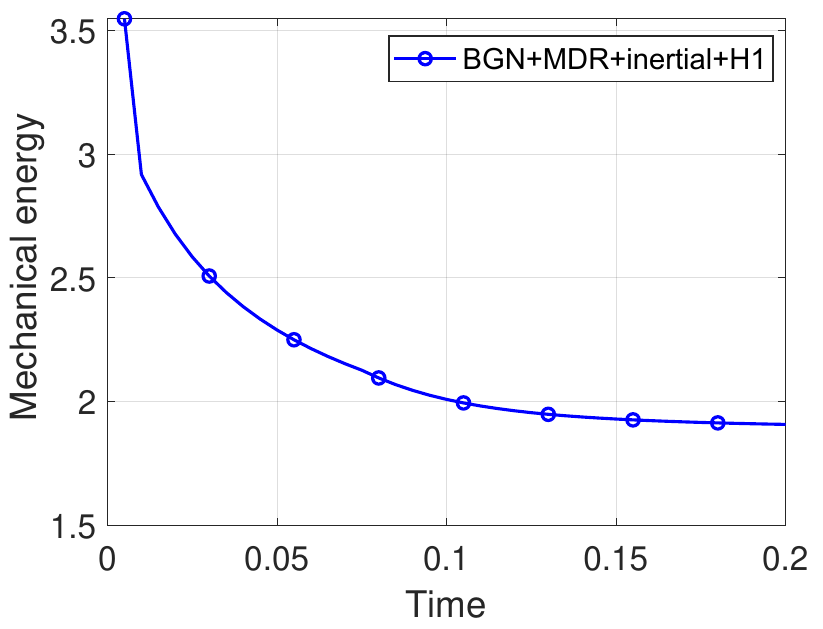}
    \end{minipage}%
    \caption{Example~2 (Case~1): velocity field around the optimized obstacle (left), and convergence histories of the original objective (middle) and mechanical energy (right).}
    \label{StokesVelObj2d}
\end{figure}

\begin{figure}[htbp]
    \centering
    \begin{minipage}[t]{0.5\textwidth}
        \centering
        \includegraphics[width=1.7in]{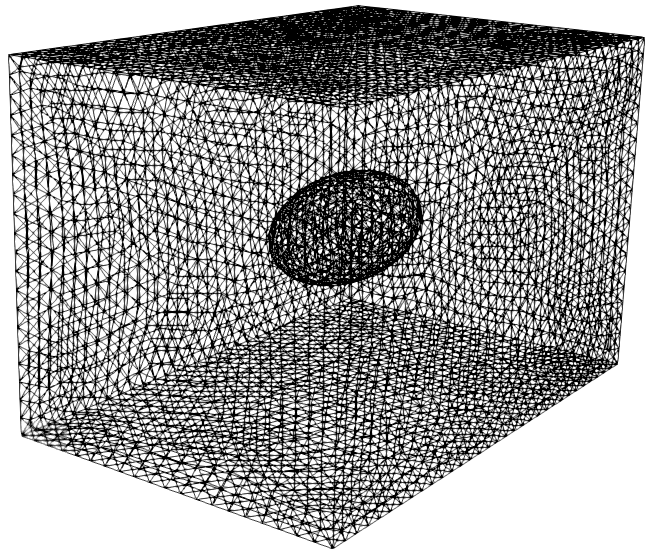}
    \end{minipage}%
    \begin{minipage}[t]{0.5\textwidth}
        \centering
        \includegraphics[width=2.5in]{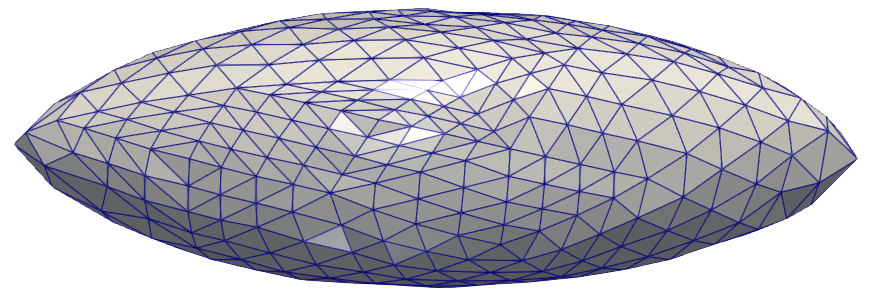}
    \end{minipage}%
    \caption{Example~2 (Case~2): initial obstacle (left) and optimized obstacle (right) in the three-dimensional drag minimization.}
    \label{StokesDrag3d}
\end{figure}

\subsection{Example 3 (Elliptic eigenvalue minimization)}
We consider the minimization of the $\ell$-th eigenvalue $\lambda_\ell$ of the elliptic operator
$-\Delta$ with homogeneous Dirichlet boundary conditions, for $\ell\in\{1,2,3,6\}$, under the
volume constraint $|\Omega|=1$, in both two and three spatial dimensions.

\noindent
\textbf{Case 1 (Minimizing the first eigenvalue in 2D).}
The initial domain is the unit square $\Omega^0=[0,1]^2$, discretized into 956 triangular elements.
It is well known that, under a volume constraint, the minimizer of $\lambda_1$ is a disk.
Figure~\ref{Eig1st2d} compares the optimized shapes obtained with the inertial BGN--MDR method
regularized by surface diffusion (left) and the classical $H^1$ gradient flow (middle).
The convergence history of the eigenvalue (original objective) shown in Fig.~\ref{Eig1st2d} (right) shows that the proposed inertial
method converges faster than the classical $H^1$ gradient flow. Moreover, the inertial BGN--MDR
method produces an optimized shape with a smooth boundary and a quasi-uniform mesh, whereas the
$H^1$ gradient flow yields a visibly non-smooth boundary with pronounced mesh clustering near the
four corners.

\noindent
\textbf{Case 2 (Minimizing eigenvalues in 3D).}
For the minimization of \(\lambda_1\) in three dimensions, the initial domain \(\Omega^0\) is chosen as a non-convex and non-smooth geometry, namely a cube with a smaller cubic corner removed; see Fig.~\ref{Eig3d1st} (left). The theoretical minimizer of \(\lambda_1\) under \(|\Omega|=1\) is the unit ball. The \(H^1\) gradient flow results in an optimized configuration with a rough surface and a highly non-uniform mesh (Fig.~\ref{Eig3d1st}, middle). In contrast, the inertial BGN--MDR method (regularized by surface diffusion) produces a shape with a smooth surface and a high-quality mesh (Fig.~\ref{Eig3d1st}, right).

The inertial BGN--MDR method also successfully computes optimized shapes for higher eigenvalues \(\ell=2,3,6\); see Fig.~\ref{Eig3dOptShapes} and Fig.~\ref{Eig3d2ndH1BGN}. Table~\ref{Eigenvalues3d} reports the corresponding optimized eigenvalues and shows good agreement with the benchmark values in~\cite{Antunes2017}, which supports the accuracy and robustness of Algorithm~\ref{alg1}. For the minimization of \(\lambda_2\), Fig.~\ref{Eig3d2ndH1BGN} compares the results obtained by the \(H^1\) gradient flow (left), the BGN--harmonic extension (middle), and the inertial BGN--MDR method (right). The \(H^1\) gradient flow exhibits severe mesh non-uniformity, with elements clustering near the initial corners and edges, leading to a highly non-smooth interface between the two spherical components. While the BGN--harmonic extension alleviates this distortion, the mesh remains insufficiently quasi-uniform, and the resulting surface displays a faceted, polyhedral appearance. By contrast, the inertial BGN--MDR method delivers the best overall performance, producing an optimized shape with a smooth surface together with a high-quality mesh.

\begin{figure}[htbp]
    \centering
    \begin{minipage}[t]{0.33\textwidth}
        \centering
        \includegraphics[width=1.5in]{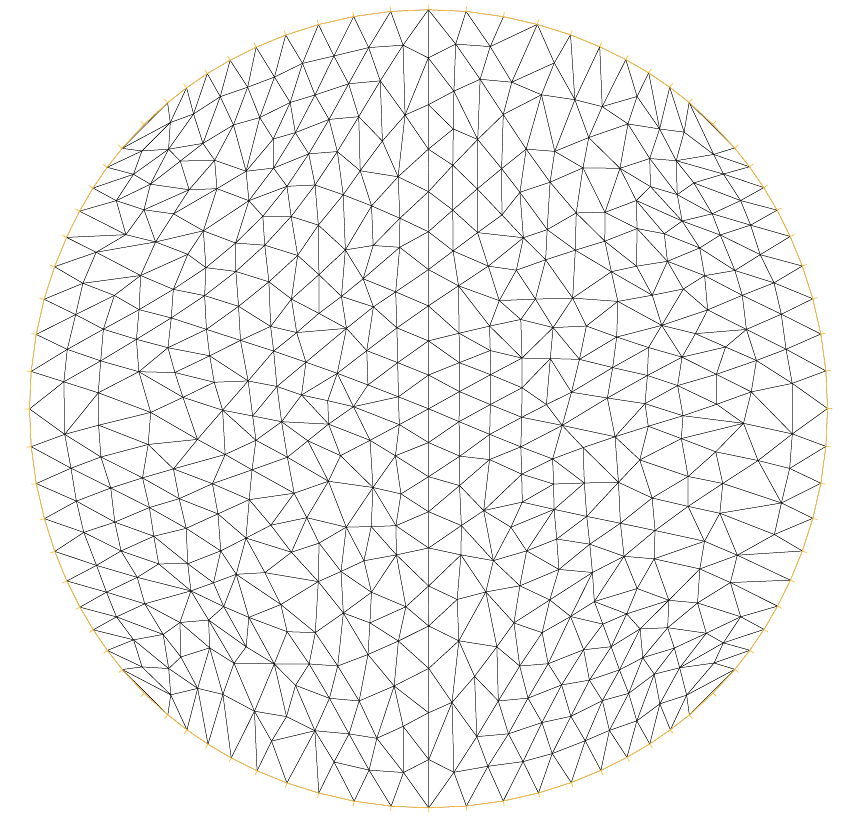}
    \end{minipage}%
    \begin{minipage}[t]{0.33\textwidth}
        \centering
        \includegraphics[width=1.5in]{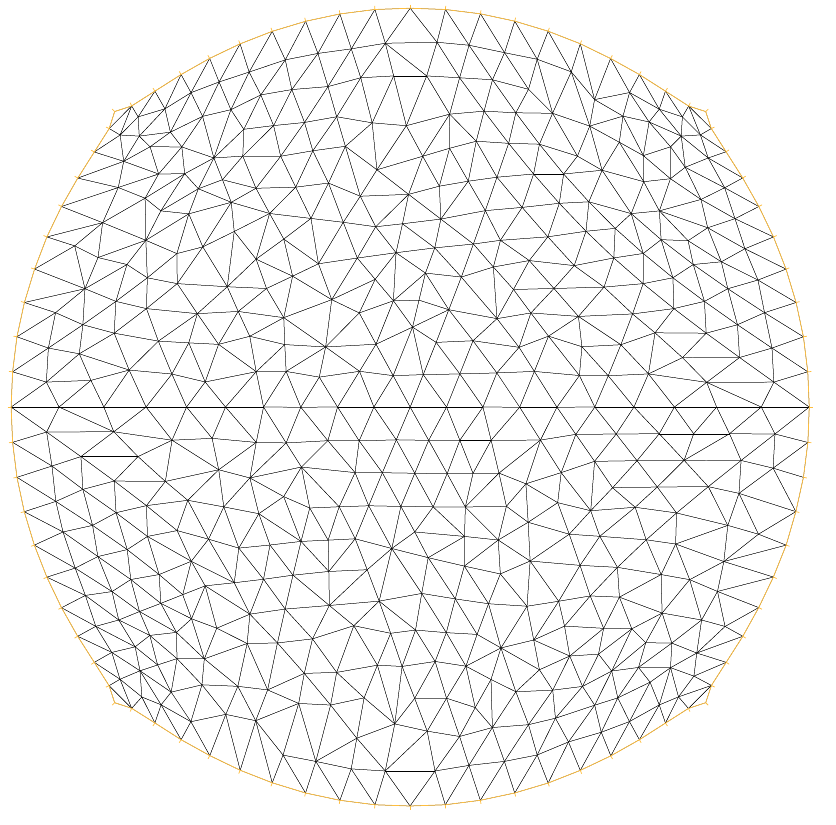}
    \end{minipage}%
    \begin{minipage}[t]{0.33\textwidth}
        \centering
        \includegraphics[width=2.0in]{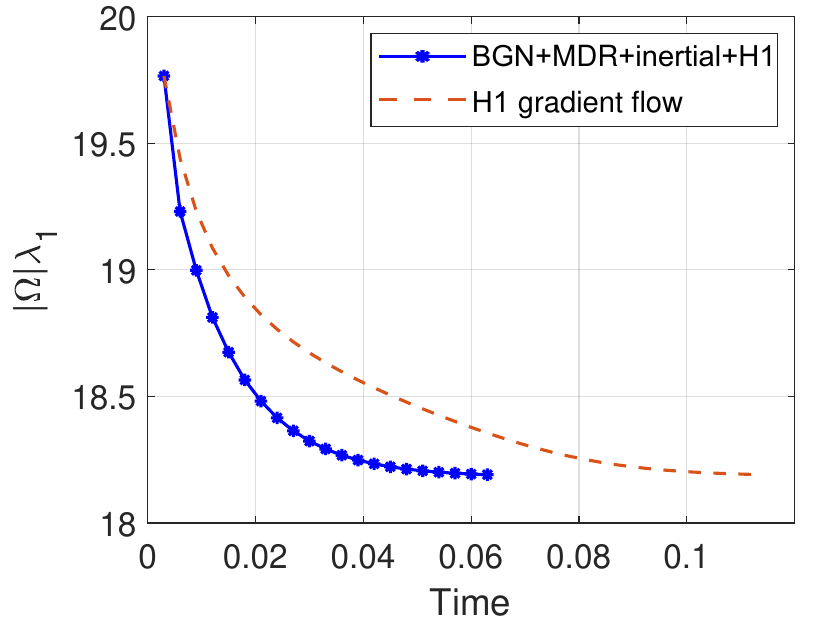}
    \end{minipage}%
    \caption{Example~3 (Case~1): optimized shapes computed by the inertial BGN--MDR method (left) and the \(H^1\) gradient flow (middle), with the convergence history of the original objective (right).}
    \label{Eig1st2d}
\end{figure}

\begin{table}[htbp]
    \centering
    \begin{tabular}{ | c | c | c || c | c | c |}
        \hline
        \(\ell\) & \(\lambda_\ell\) & \(\lambda_\ell\) \cite{Antunes2017}
        & \(\ell\) & \(\lambda_\ell\) & \(\lambda_\ell\) \cite{Antunes2017} \\ \hline
        1 & \textbf{25.96} & 25.65 & 3 & \textbf{49.94} & 49.17 \\ \hline
        2 & \textbf{41.27} & 40.72 & 6 & \textbf{72.31} & 73.05 \\ \hline
    \end{tabular}
    \caption{Example~3 (Case~2): comparison of optimized eigenvalues in three dimensions.}
    \label{Eigenvalues3d}
\end{table}

\begin{figure}[htbp]
    \centering
    \begin{minipage}[t]{0.33\textwidth}
        \centering
        \includegraphics[width=1.5in]{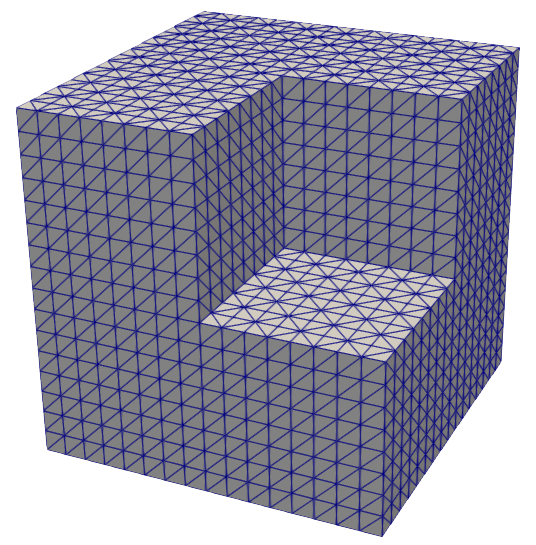}
    \end{minipage}%
    \begin{minipage}[t]{0.33\textwidth}
        \centering
        \includegraphics[width=1.5in]{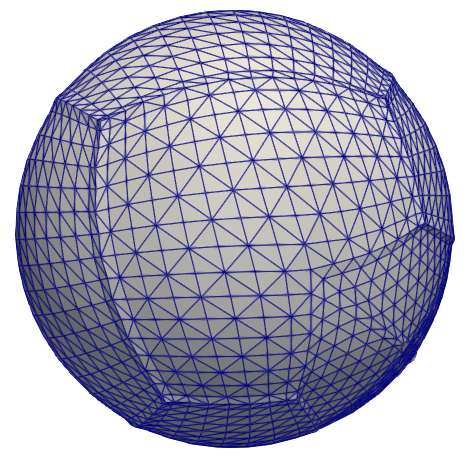}
    \end{minipage}%
    \begin{minipage}[t]{0.33\textwidth}
        \centering
        \includegraphics[width=1.5in]{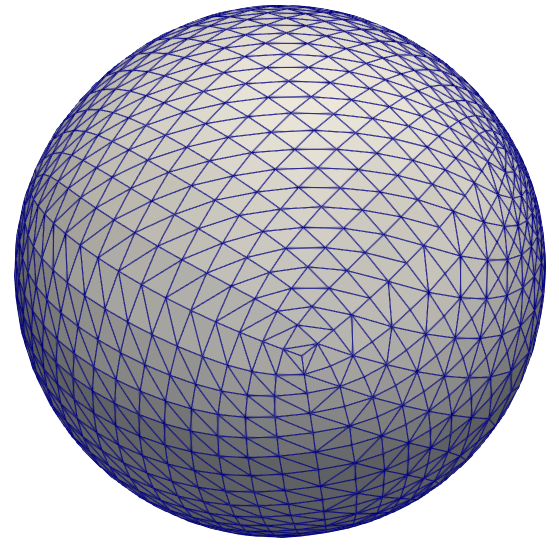}
    \end{minipage}%
    \caption{Example~3 (Case~2): initial non-convex and non-smooth shape (left), and optimized shapes produced by the \(H^1\) gradient flow (middle) and the inertial BGN--MDR method (right) for the minimization of \(\lambda_1\) in 3D.}
    \label{Eig3d1st}
\end{figure}

\begin{figure}[htbp]
    \centering
    \begin{minipage}[t]{0.33\textwidth}
        \centering
        \includegraphics[width=1.7in]{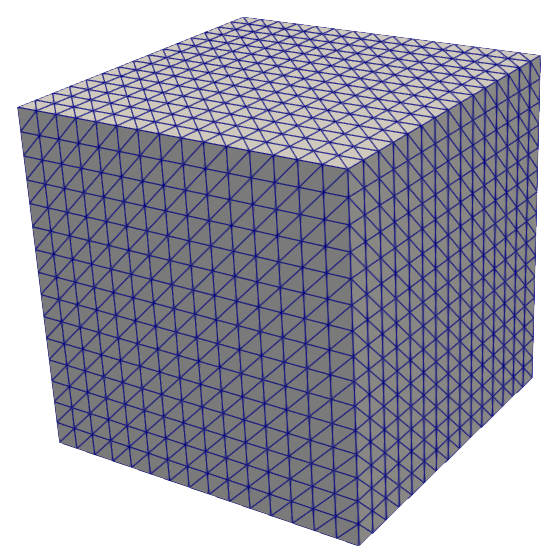}
    \end{minipage}%
    \begin{minipage}[t]{0.33\textwidth}
        \centering
        \includegraphics[width=1.9in]{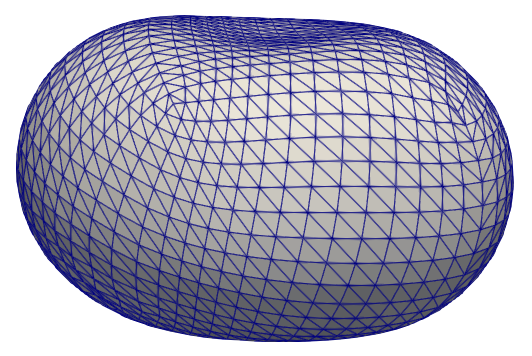}
    \end{minipage}%
    \begin{minipage}[t]{0.33\textwidth}
        \centering
        \includegraphics[width=1.4in]{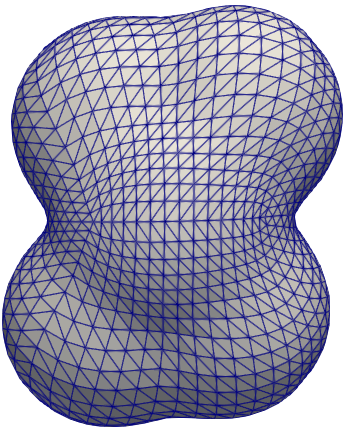}
    \end{minipage}%
    \caption{Example~3 (Case~2): initial shape (left) and optimized shapes for \(\ell=3\) (middle) and \(\ell=6\) (right).}
    \label{Eig3dOptShapes}
\end{figure}

\begin{figure}[htbp]
    \centering
    \begin{minipage}[t]{0.33\textwidth}
        \centering
        \includegraphics[width=1.8in]{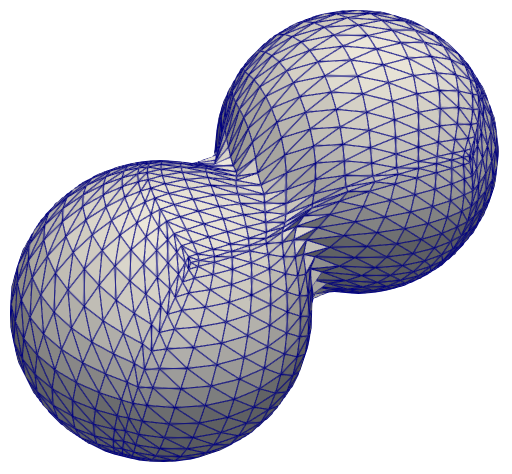}
    \end{minipage}%
    \begin{minipage}[t]{0.33\textwidth}
        \centering
        \includegraphics[width=1.8in]{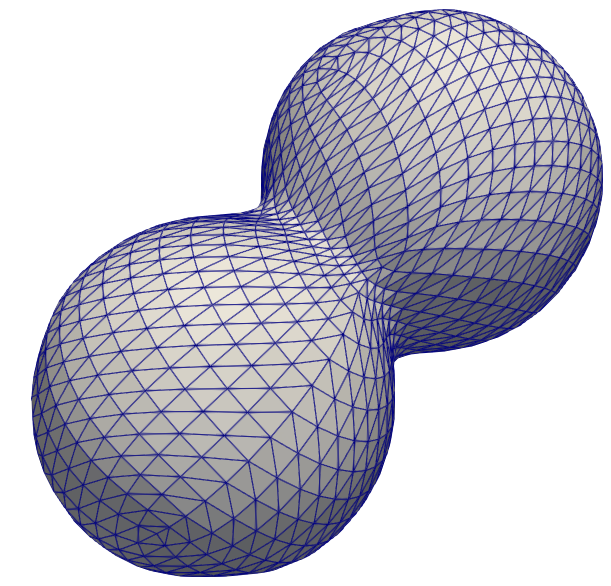}
    \end{minipage}%
    \begin{minipage}[t]{0.33\textwidth}
        \centering
        \includegraphics[width=1.8in]{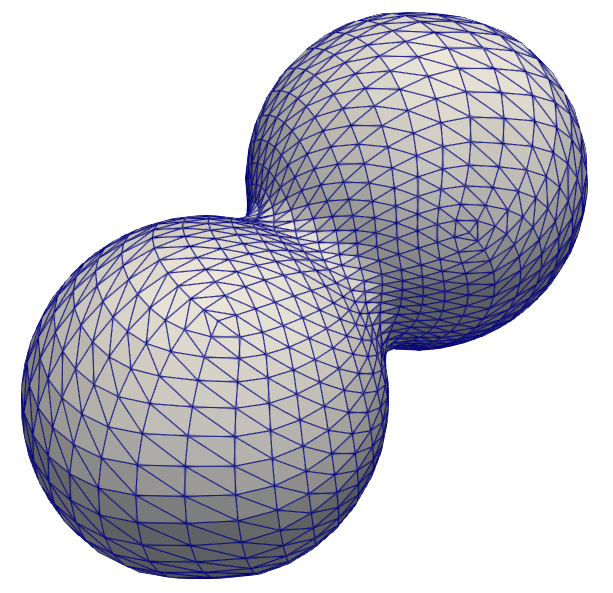}
    \end{minipage}%
    \caption{Example~3 (Case~2): optimized meshes obtained by the \(H^1\) gradient flow (left), the BGN--harmonic extension (middle), and the inertial BGN--MDR method (right) for the minimization of \(\lambda_2\) in 3D.}
    \label{Eig3d2ndH1BGN}
\end{figure}

\subsection{Example 4 (Surface hole filling)}
We present numerical simulations to demonstrate surface hole filling via the Willmore flow~\eqref{willmore-flow}. Using a car surface with two holes as a representative test case (see Fig.~\ref{Car1}, left), we compare the standard BGN scheme~\cite{BarrettGarckeNuernberg2008SISC} with the proposed inertial--MDR scheme~\eqref{numer-willmore}, with particular emphasis on the resulting mesh quality.

We set $\tau=0.01$ and $\epsilon_0=0.1$. Starting from the initial configuration
(Fig.~\ref{Car1IniOptMesh}, row~1 left), we deliberately choose incompatible initial data in the
sense that the conormal along the patch boundary does not match that of the prescribed exterior
surface. The BGN scheme then yields a numerically steady surface with noticeably inferior mesh
quality (Fig.~\ref{Car1IniOptMesh}, row~1 middle), characterized by distorted elements and a
strongly non-uniform vertex distribution. By contrast, the inertial--MDR scheme produces a surface
with a quasi-uniform triangulation and substantially improved element quality
(Fig.~\ref{Car1IniOptMesh}, row~1 right). The highlighted feature lines in
Fig.~\ref{Car1IniOptMesh} (row~2) further show that the inertial--MDR reconstruction is
considerably smoother. Finally, the normal field and Gaussian curvature of the recovered patch,
together with the Willmore-energy decay shown in Fig.~\ref{Car1} (right), demonstrate the
effectiveness and computational performance of the inertial--MDR scheme~\eqref{numer-willmore}:
the inertial term accelerates convergence to a lower Willmore energy compared with the MDR method
without inertia, while maintaining good mesh quality.

\begin{figure}[htbp]
    \centering
    \begin{minipage}[t]{0.5\textwidth}
        \centering
        \includegraphics[width=2.5in]{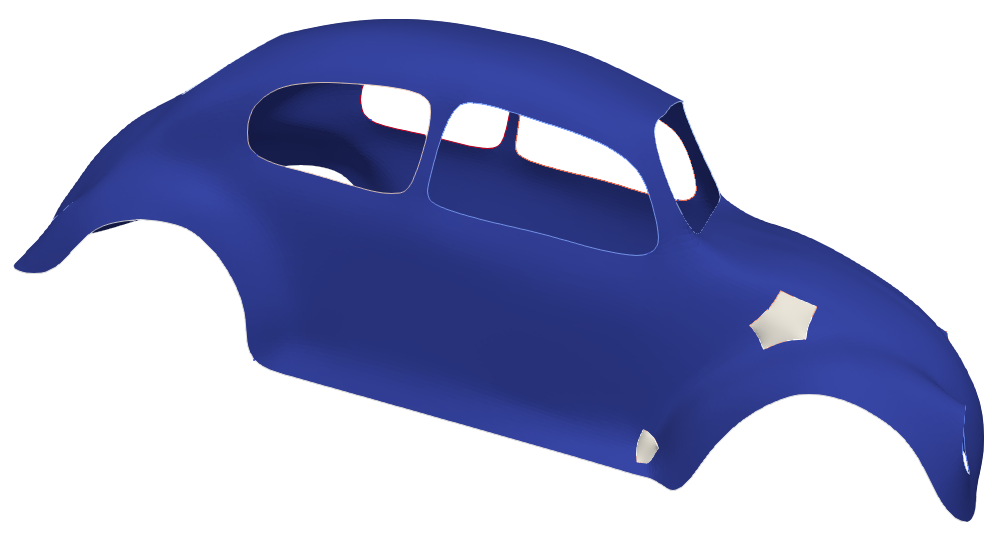}
    \end{minipage}%
    \begin{minipage}[t]{0.5\textwidth}
        \centering
        \includegraphics[width=2.2in]{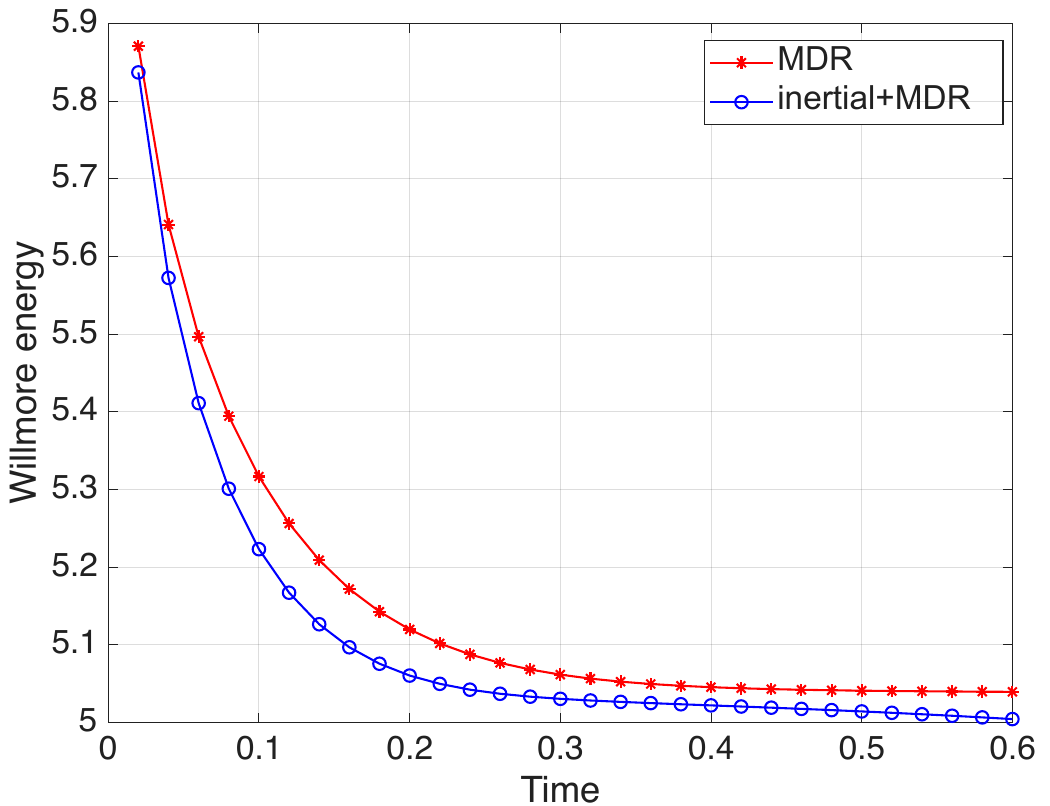}
    \end{minipage}%
    \caption{Example~4: a car surface with holes (left) and convergence history of the Willmore energy (right) computed by the inertial--MDR scheme.}
    \label{Car1}
\end{figure}

\begin{figure}[htbp]
    \centering
    \begin{minipage}[t]{0.33\textwidth}
        \centering
        \includegraphics[width=1.9in]{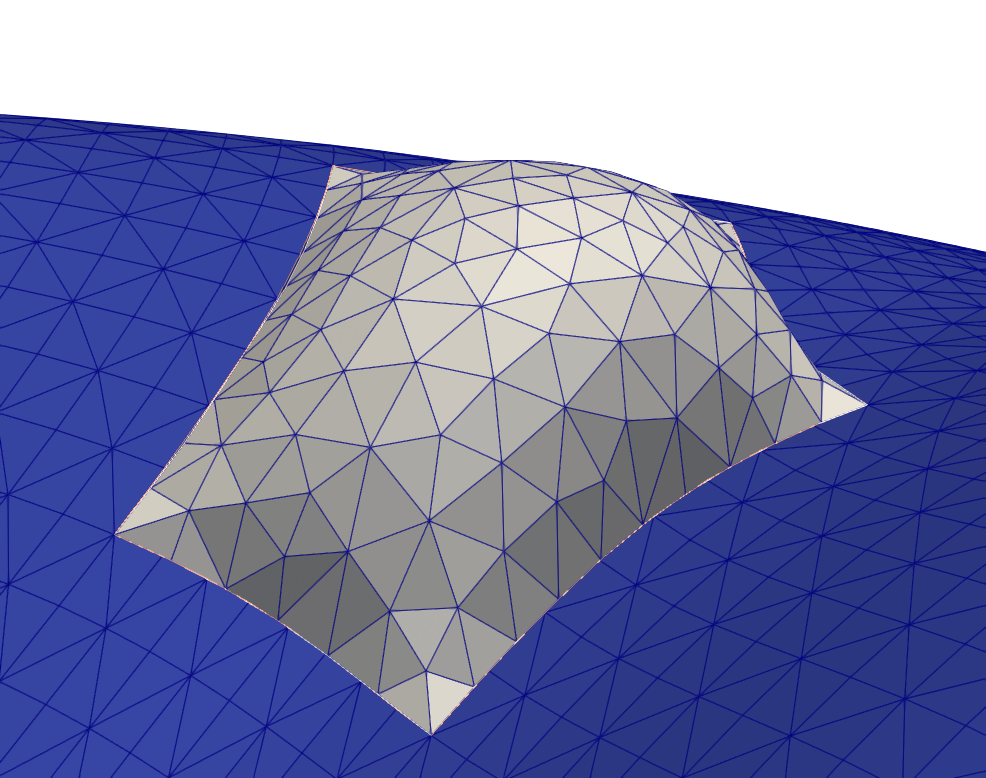}
    \end{minipage}%
    \begin{minipage}[t]{0.33\textwidth}
        \centering
        \includegraphics[width=1.9in]{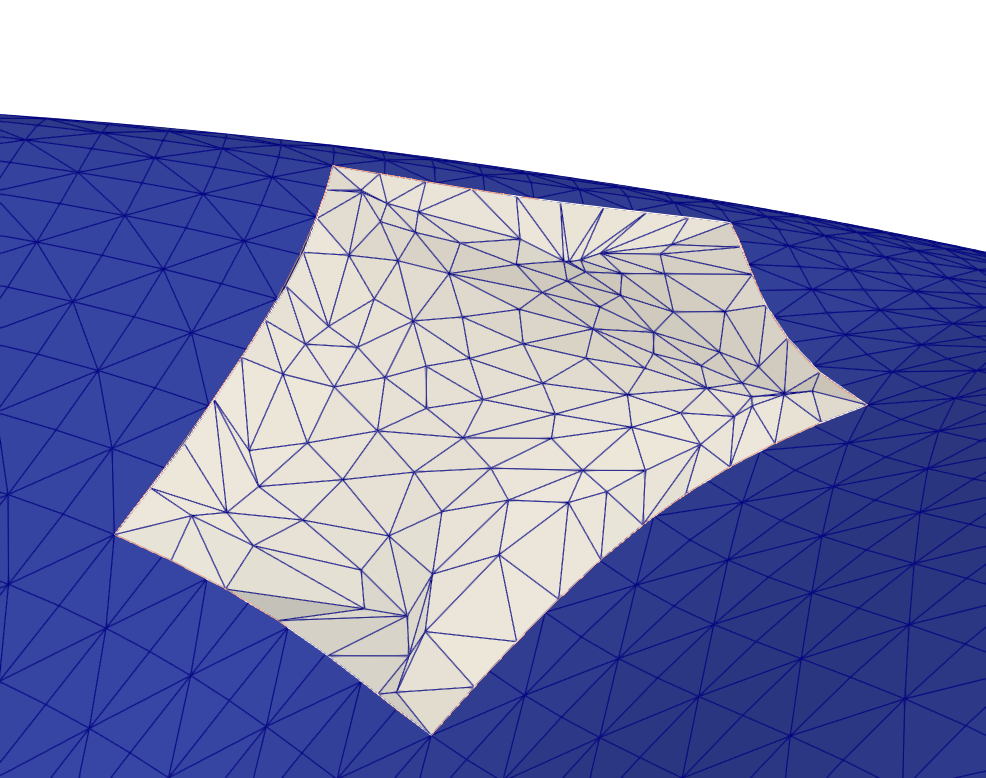}
    \end{minipage}%
    \begin{minipage}[t]{0.33\textwidth}
        \centering
        \includegraphics[width=1.9in]{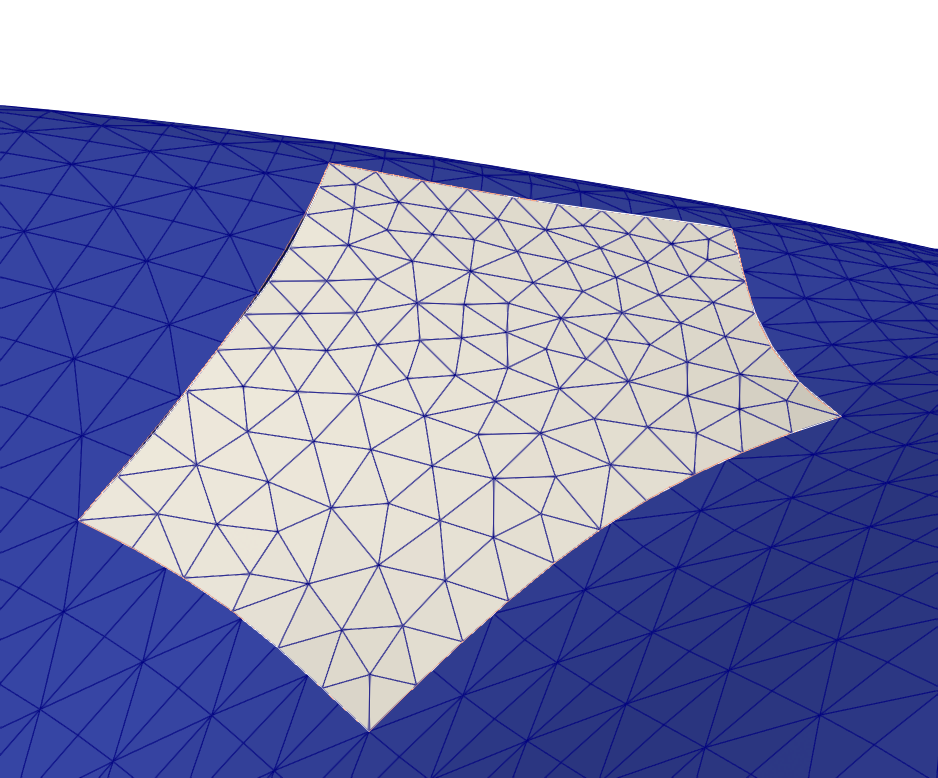}
    \end{minipage}%
    \\[0.3em]
    \begin{minipage}[t]{0.25\textwidth}
        \centering
        \includegraphics[width=1.25in]{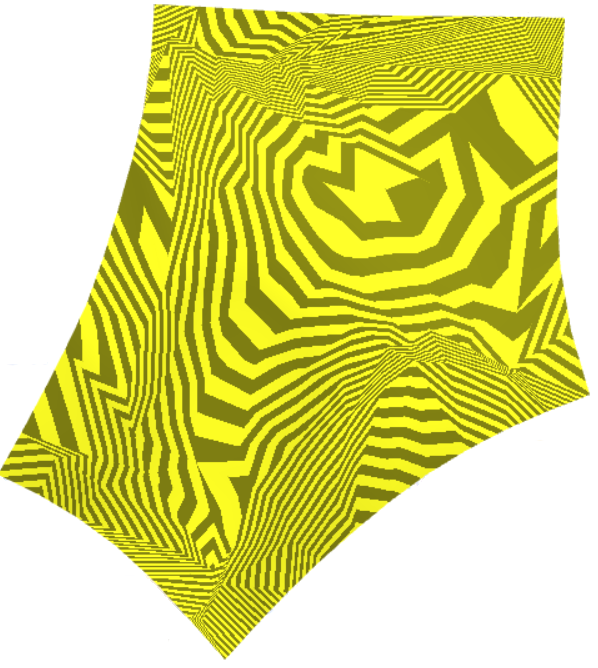}
    \end{minipage}%
    \begin{minipage}[t]{0.25\textwidth}
        \centering
        \includegraphics[width=1.25in]{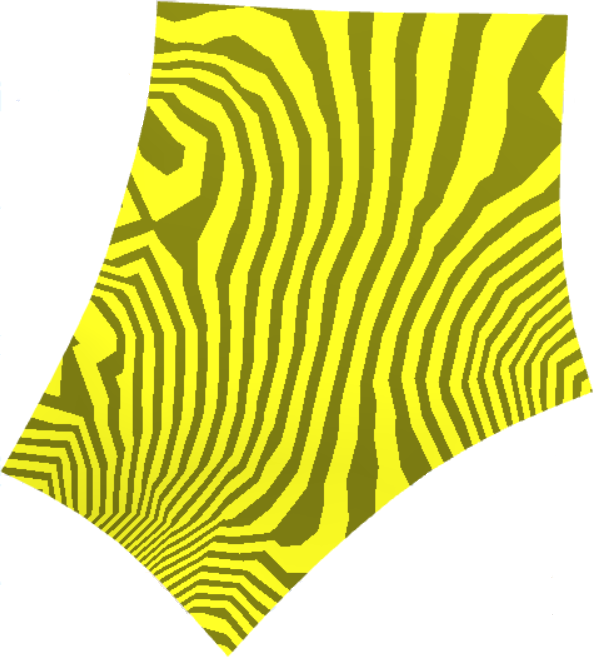}
    \end{minipage}%
    \begin{minipage}[t]{0.25\textwidth}
        \centering
        \includegraphics[width=1.5in]{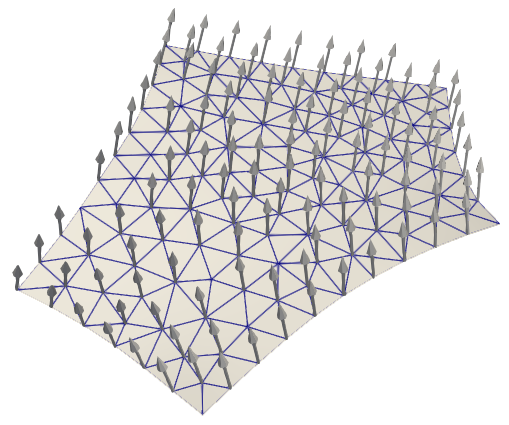}
    \end{minipage}%
    \begin{minipage}[t]{0.25\textwidth}
        \centering
        \includegraphics[width=1.5in]{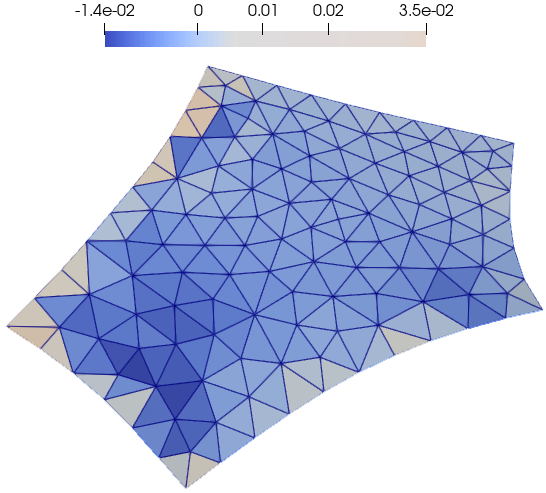}
    \end{minipage}%
    \caption{Example~4: Row~1 shows the initial patch mesh (left) and the steady-state meshes obtained by the BGN scheme (middle) and the inertial--MDR scheme (right). Row~2 shows highlighted feature lines for the BGN reconstruction (left) and the inertial--MDR reconstruction (middle left), as well as the normal field (middle right) and Gaussian curvature (right) of the inertial--MDR recovered patch.}
    \label{Car1IniOptMesh}
\end{figure}



\section{Conclusion}

In this work, we proposed an inertial--MDR framework for PDE--constrained shape optimization that
accelerates convergence and attains lower original objective values while preserving high mesh quality
without remeshing. By integrating inertial dynamics with surface--diffusion regularization, the
method computes the boundary velocity via the BGN formulation and propagates the prescribed motion
into the bulk through the MDR method. As a result, the approach remains robust for non-smooth and
non-convex initial geometries. Numerical experiments on shape reconstruction, Stokes drag
minimization, and elliptic eigenvalue minimization demonstrate faster convergence, lower original objective
values, and improved mesh quality compared with classical first--order $H^1$ shape gradient flows
and the BGN--harmonic extension approach.

We further extended the approach to an inertial--MDR scheme for Willmore--driven surface hole
filling under clamped boundary conditions. In this setting, the method accelerates convergence
toward lower objective values and preserves high mesh quality without remeshing or
reinitialization, even when initialized from incompatible data. 
Future work will focus on a
rigorous convergence analysis of the fully discrete schemes and on extending the proposed
framework to broader classes of geometric flows and shape optimization problems.



\bibliographystyle{plain}  
\bibliography{references1}

@incollection{Antunes2017,
  author = {Antunes, P. R. S. and Oudet, E.},
  title = {Numerical {R}esults for {E}xtremal {P}roblem for {E}igenvalues of the {L}aplacian},
  booktitle = {Shape {O}ptimization and {S}pectral {T}heory},
  pages = {398--411},
  publisher = {De Gruyter Open},
  address = {Warsaw},
  year = {2017}
}

@book{DelZol,
  author = {Delfour, M. and Zol\'{e}sio, J. P.},
  title = {Shapes and Geometries: Metrics, Analysis, Differential Calculus, and Optimization},
  edition = {2nd},
  publisher = {SIAM},
  address = {Philadelphia},
  year = {2011}
}

@book{SokolowskiZolesio1992,
  author = {Sokolowski, J. and Zol\'{e}sio, J. P.},
  title = {Introduction to Shape Optimization: Shape Sensitivity Analysis},
  publisher = {Springer},
  address = {Heidelberg},
  year = {1992}
}

@article{Burger2003,
  author = {Burger, M.},
  title = {A framework for the construction of level set methods for shape optimization and reconstruction},
  journal = {Interfaces and Free Boundaries},
  volume = {5},
  year = {2003},
  pages = {301--329}
}

@article{Gournay2006,
  author = {de Gournay, F.},
  title = {Velocity extension for the level-set method and multiple eigenvalues in shape optimization},
  journal = {SIAM Journal on Control and Optimization},
  volume = {45},
  year = {2006},
  pages = {343--367}
}

@article{CR2018,
  author = {Iglesias, J. A. and Sturm, K. and Wechsung, F.},
  title = {Two-{D}imensional shape optimization with nearly conformal transformations},
  journal = {SIAM Journal on Scientific Computing},
  volume = {40},
  year = {2018},
  pages = {3807--3830}
}

@article{Mullins1957,
  author = {Mullins, W. W.},
  title = {Theory of thermal grooving},
  journal = {Journal of Applied Physics},
  volume = {28},
  year = {1957},
  pages = {333--339}
}

@article{SchulzCMAM2016,
  author = {Schulz, V. and Siebenborn, M.},
  title = {Computational comparison of surface metrics for {PDE} constrained shape optimization},
  journal = {Computational Methods in Applied Mathematics},
  volume = {16},
  year = {2016},
  pages = {485--496}
}

@article{Schulz2016SICON,
  author = {Schulz, V. and Siebenborn, M. and Welker, K.},
  title = {Efficient {PDE} constrained shape optimization based on {S}teklov-{P}oincar\'{e}-type metrics},
  journal = {SIAM Journal on Optimization},
  volume = {26},
  year = {2016},
  pages = {2800--2819}
}

@article{Hiptmair2015BIT,
  author = {Hiptmair, R. and Paganini, A. and Sargheini, S.},
  title = {Comparison of approximate shape gradients},
  journal = {BIT Numerical Mathematics},
  volume = {55},
  year = {2015},
  pages = {459--485}
}

@book{Walker2015,
  author = {Walker, S. W.},
  title = {The Shapes of Things: A Practical Guide to Differential Geometry and the Shape Derivative},
  publisher = {SIAM},
  year = {2015}
}

@article{ZhuCMAME,
  author = {Zhu, S. and Gao, Z.},
  title = {Convergence analysis of mixed finite element approximations to shape gradients in the {S}tokes equation},
  journal = {Computer Methods in Applied Mechanics and Engineering},
  volume = {343},
  year = {2019},
  pages = {127--150}
}

@article{ZhuEig,
  author = {Zhu, S.},
  title = {Effective shape optimization of {L}aplace eigenvalue problems using domain expressions of {E}ulerian derivatives},
  journal = {Journal of Optimization Theory and Applications},
  volume = {176},
  year = {2018},
  pages = {17--34}
}

@book{MohammadiPironneau2010,
  author = {Mohammadi, B. and Pironneau, O.},
  title = {Applied Shape Optimization for Fluids},
  edition = {2nd},
  publisher = {Oxford University Press},
  address = {Oxford},
  year = {2010}
}

@incollection{AllaireDapognyJouve2021,
  author = {Allaire, G. and Dapogny, C. and Jouve, F.},
  title = {Shape and topology optimization},
  booktitle = {Geometric {P}artial {D}ifferential {E}quations, {P}art {II}},
  series = {Handbook of Numerical Analysis},
  volume = {22},
  pages = {1--132},
  publisher = {Elsevier},
  address = {Amsterdam},
  year = {2021}
}

@book{HaslingerMakinen2003,
  author = {Haslinger, J. and M\"{a}kinen, R.},
  title = {Introduction to Shape Optimization: Theory, Approximation, and Computation},
  publisher = {SIAM},
  address = {Philadelphia},
  year = {2003}
}

@book{HaslingerNeittaanmaki1996,
  author = {Haslinger, J. and Neittaanm\"{a}ki, P.},
  title = {Finite Element Approximation for Optimal Shape, Material and Topology Design},
  edition = {2nd},
  publisher = {John Wiley \& Sons},
  address = {Chichester},
  year = {1996}
}

@book{HaugChoiKomkov1986,
  author = {Haug, E. J. and Choi, K. K. and Komkov, V.},
  title = {Design Sensitivity Analysis of Structural Systems},
  publisher = {Academic Press},
  address = {Orlando, FL},
  year = {1986}
}

@article{DoganMorinNochettoVerani2007,
  author = {Dogan, G. and Morin, P. and Nochetto, R. H. and Verani, M.},
  title = {Discrete gradient flows for shape optimization and applications},
  journal = {Computer Methods in Applied Mechanics and Engineering},
  volume = {196},
  year = {2007},
  pages = {3898--3914}
}

@article{GongLiRao2024,
  author = {Gong, W. and Li, B. and Rao, Q.},
  title = {Convergent evolving finite element approximations of boundary evolution under shape gradient flow},
  journal = {IMA Journal of Numerical Analysis},
  volume = {44},
  year = {2024},
  pages = {2667--2697}
}

@article{GongZhu2021,
  author = {Gong, W. and Zhu, S.},
  title = {On discrete shape gradients of boundary type for {PDE}-constrained shape optimization},
  journal = {SIAM Journal on Numerical Analysis},
  volume = {59},
  number = {3},
  year = {2021},
  pages = {1510--1541}
}

@book{HenrotPierre2018,
  author = {Henrot, A. and Pierre, M.},
  title = {Shape Variation and Optimization},
  series = {EMS Tracts in Mathematics},
  volume = {28},
  year = {2018}
}

@article{Burman2017,
  author = {Burman, E. and Elfverson, D. and Hansbo, P. and Larson, M.G. and Larsson, K.},
  title = {A cut finite element method for the {B}ernoulli free boundary value problem},
  journal = {Computer Methods in Applied Mechanics and Engineering},
  volume = {317},
  year = {2017},
  pages = {598--618}
}

@article{Guo2020,
  author = {Guo, R. and Lin, T. and Lin, Y.},
  title = {Recovering elastic inclusions by shape optimization methods with immersed finite elements},
  journal = {Journal of Computational Physics},
  volume = {404},
  year = {2020},
  pages = {109123},
  note = {24 pp.}
}

@article{Dede2012,
  author = {Ded\`e, L. and Borden, M. J.  and Hughes, T. J. R. },
  title = {Isogeometric analysis for topology optimization with a phase field model},
  journal = {Archives of Computational Methods in Engineering},
  volume = {19},
  year = {2012},
  pages = {427--465}
}

@article{Wang2018,
  author = {Wang, C. and Xia, S. and Wang, X. and Qian, X.},
  title = {Isogeometric shape optimization on triangulations},
  journal = {Computer Methods in Applied Mechanics and Engineering},
  volume = {331},
  year = {2018},
  pages = {585--622}
}

@article{Morin2012,
  author = {Morin, P. and Nochetto, R. H. and Pauletti, M. S. and Verani, M.},
  title = {Adaptive {FEM} for shape optimization},
  journal = {ESAIM: Control, Optimisation and Calculus of Variations},
  volume = {18},
  year = {2012},
  pages = {1122--1149}
}

@article{Eppler2006,
  author = {Eppler, K. and Harbrecht, H.},
  title = {Coupling of {FEM} and {BEM} in shape optimization},
  journal = {Numerische Mathematik},
  volume = {104},
  year = {2006},
  pages = {47--68}
}

@article{Ma2022,
  author = {Ma, C. and Zhang, Q. and Zheng, W.},
  title = {A fourth-order unfitted characteristic {FEM} for solving the advection--diffusion equation on time-varying domains},
  journal = {SIAM Journal on Numerical Analysis},
  volume = {60},
  year = {2022},
  pages = {2203--2224}
}

@article{Boffi2004,
  author = {Boffi, D. and Gastaldi, L.},
  title = {Stability and geometric conservation laws for arbitrary {L}agrangian--{E}ulerian formulations},
  journal = {Computer Methods in Applied Mechanics and Engineering},
  volume = {193},
  year = {2004},
  pages = {4717--4739}
}

@article{Badia2006,
  author = {Badia, S. and Codina, R.},
  title = {Analysis of a stabilized finite element approximation of the transient convection--diffusion equation using an arbitrary {L}agrangian--{E}ulerian framework},
  journal = {SIAM Journal on Numerical Analysis},
  volume = {44},
  year = {2006},
  pages = {2159--2197}
}

@article{ElliottRanner2021,
  author = {Elliott, C. M. and Ranner, T.},
  title = {A unified theory for continuous-in-time evolving finite element space approximations to partial differential equations in evolving domains},
  journal = {IMA Journal of Numerical Analysis},
  volume = {41},
  number = {3},
  year = {2021},
  pages = {1696--1845}
}

@article{Edelmann2022,
  author = {Edelmann, D.},
  title = {Finite element analysis for a diffusion equation on a harmonically evolving domain},
  journal = {IMA Journal of Numerical Analysis},
  volume = {42},
  year = {2022},
  pages = {1866--1901}
}

@article{Gastaldi2001,
  author = {Gastaldi, L.},
  title = {A priori error estimates for the arbitrary {L}agrangian {E}ulerian formulation with finite elements},
  journal = {East--West Journal of Numerical Mathematics},
  volume = {9},
  year = {2001},
  pages = {123--156}
}

@phdthesis{Nobile2001,
  author = {Nobile, F.},
  title = {Numerical approximation of fluid-structure interaction problems with application to haemodynamics},
  school = {\'Ecole Polytechnique F\'ed\'erale de Lausanne},
  address = {Switzerland},
  year = {2001}
}

@article{ElliottStyles2012,
  author = {Elliott, C. M. and Styles, V.},
  title = {An arbitrary {L}agrangian--{E}ulerian {ESFEM} for solving {PDE}s on evolving surfaces},
  journal = {Milan Journal of Mathematics},
  volume = {80},
  year = {2012},
  pages = {469--501}
}

@article{ElliottVenkataraman2015,
  author = {Elliott, C. M. and Venkataraman, C.},
  title = {Error analysis for an arbitrary {L}agrangian--{E}ulerian evolving surface {FEM}},
  journal = {Numerical Methods for Partial Differential Equations},
  volume = {31},
  year = {2015},
  pages = {459--499}
}

@article{KovacsGuerra2018,
  author = {Kov\'acs, B. and Guerra, C. A. Power},
  title = {Higher order time discretizations with arbitrary {L}agrangian--{E}ulerian finite elements for parabolic problems on evolving surfaces},
  journal = {IMA Journal of Numerical Analysis},
  volume = {38},
  year = {2018},
  pages = {460--494}
}

@article{2ndFlowSISC2025,
  author = {Chen, H. and Dong, G. and Iglesias, J. A.  and Liu, W. and Xie, Z.},
  title = {Second-order flows for approaching stationary points of a class of nonconvex energies via convex-splitting schemes},
  journal = {SIAM Journal on Scientific Computing},
  volume = {47},
  year = {2025},
  pages = {A1604--1627}
}

@article{Boyd2016,
  author = {Su, W. and Boyd, S. and Cand\'{e}s, E. J.},
  title = {A differential equation for modeling {N}esterov's accelerated gradient method: {T}heory and insights},
  journal = {Journal of Machine Learning Research},
  volume = {17},
  year = {2016},
  pages = {153}
}

@article{Edvardsson2012,
  author = {Edvardsson, S. and Gulliksson, M. and Persson, J.},
  title = {The dynamical functional particle method: An approach for boundary value problems},
  journal = {Journal of Applied Mechanics},
  volume = {79},
  year = {2012},
  pages = {021012}
}

@article{Nesterov1983,
  author = {Nesterov, Y.},
  title = {A method of solving a convex programming problem with convergence rate $O(1/k^2)$},
  journal = {Soviet Mathematics Doklady},
  volume = {27},
  year = {1983},
  pages = {372--376}
}

@article{AttouchMP2018,
  author = {Attouch, H. and Chbani, Z. and Peypouquet, J. and Redont, P.},
  title = {Fast convergence of inertial dynamics and algorithms with asymptotic vanishing viscosity},
  journal = {Mathematical Programming},
  volume = {168},
  year = {2018}
}

@article{ChenDong2023JCP,
  author = {Chen, H. and Dong, G. and Liu, W. and Xie, Z.},
  title = {Second-order flows for computing the ground states of rotating {B}ose-{E}instein condensates},
  journal = {Journal of Computational Physics},
  volume = {475},
  year = {2023},
  pages = {111872}
}

@article{BaoZhao2021,
  author  = {Bao, W. and Zhao, Q.},
  title   = {A structure-preserving parametric finite element method for surface diffusion},
  journal = {SIAM Journal on Numerical Analysis},
  year    = {2021},
  volume  = {59},
  pages   = {2775--2799},
}

@article{BaenschMorinNochetto2005,
  author  = {B{\"a}nsch, E. and Morin, P. and Nochetto, R. H.},
  title   = {A finite element method for surface diffusion: {T}he parametric case},
  journal = {Journal of Computational Physics},
  year    = {2005},
  volume  = {203},
  pages   = {321--343},
}

@article{BarrettGarckeNuernberg2007,
  author  = {Barrett, J. W. and Garcke, H. and N{\"u}rnberg, R.},
  title   = {A parametric finite element method for fourth order geometric evolution equations},
  journal = {Journal of Computational Physics},
  year    = {2007},
  volume  = {222},
  pages   = {441--467}
}

@article{BarrettGarckeNuernberg2008JCP,
  author  = {Barrett, J. W. and Garcke, H. and N{\"u}rnberg, R.},
  title   = {On the parametric finite element approximation of evolving hypersurfaces in $\mathbb{R}^3$},
  journal = {Journal of Computational Physics},
  year    = {2008},
  volume  = {227},
  pages   = {4281--4307}
}

@article{BarrettGarckeNuernberg2008SISC,
  author  = {Barrett, J. W. and Garcke, H. and N{\"u}rnberg, R.},
  title   = {Parametric approximation of {W}illmore flow and related geometric evolution equations},
  journal = {SIAM Journal on Scientific Computing},
  year    = {2008},
  volume  = {31},
  pages   = {225--253}
}

@article{LiBao2021,
  author  = {Li, Y. and Bao, W.},
  title   = {An energy-stable parametric finite element method for anisotropic surface diffusion},
  journal = {Journal of Computational Physics},
  year    = {2021},
  volume  = {446},
  pages   = {110658}
}

@article{BaiHuLi2024,
  author  = {Bai, G. and Hu, J. and Li, B.},
  title   = {A convergent evolving finite element method with artificial tangential motion for surface evolution under a prescribed velocity field},
  journal = {SIAM Journal on Numerical Analysis},
  year    = {2024},
  volume  = {62},
  pages   = {2172--2195}
}

@article{LiXiaYang2023,
  author  = {Li, B. and Xia, Y. and Yang, Z.},
  title   = {Optimal convergence of arbitrary {L}agrangian--{E}ulerian iso-parametric finite element methods for parabolic equations in an evolving domain},
  journal = {IMA Journal of Numerical Analysis},
  year    = {2023},
  volume  = {43},
  pages   = {501--534}
}

@misc{garcke2025energystableparametricfiniteelement,
      title={An energy-stable parametric finite element method for {W}illmore flow with normal-tangential velocity splitting}, 
      author={Garcke, H. and Nürnberg, R. and Zhao, Q.},
      year={2025},
      eprint={2507.00193},
      archivePrefix={arXiv},
      primaryClass={math.NA},
      url={https://arxiv.org/abs/2507.00193}, 
}

@article{Gao2024_rotation,
  title = {A {S}tabilized arbitrary {L}agrangian-{E}ulerian sliding interface method for fluid-structure
  interaction with a rotating rigid structure},
  author = {Gao, Y. and Hu, J. and Li, B.},
  date = {},
  volume = {},
  number = {},
  pages = {},
  doi = {},
  langid = {english},
  journal = {SIAM J. Sci. Comput.},
  year = {2024},
  dimensions = {true},
}

@book{hoschek1996fundamentals,
  title={Fundamentals of Computer Aided Geometric Design},
  author={Hoschek, J. and Lasser, D.},
  isbn={9781568810072},
  lccn={92038728},
  series={Ak Peters Series},
  url={https://books.google.com.hk/books?id=kYG17apHcrYC},
  year={1996},
  publisher={Taylor \& Francis}
}

@book{carmo,
  author    = {do Carmo, M. P.},
  title     = {Differential Geometry of Curves and Surfaces},
  publisher = {Prentice--Hall},
  address   = {Englewood Cliffs, NJ},
  year      = {1976},
  isbn      = {0132125897},
}

@book{JohnLEE,
  author    = {Lee, J. M.},
  title     = {Introduction to Topological Manifolds},
  series    = {Graduate Texts in Mathematics},
  volume    = {202},
  edition   = {2},
  publisher = {Springer},
  year      = {2011},
  isbn      = {978-1-4419-7939-1},
  isbn      = {978-1-4419-7940-7},
  doi       = {10.1007/978-1-4419-7940-7},
}

@book{missy,
  author    = {Massey, W. S.},
  title     = {A Basic Course in Algebraic Topology},
  series    = {Graduate Texts in Mathematics},
  volume    = {127},
  edition   = {1},
  publisher = {Springer},
  year      = {1991},
  isbn      = {978-0-387-97430-9},
  isbn      = {978-1-4939-9063-4},
  doi       = {10.1007/978-1-4939-9063-4},
}

\end{document}